\documentclass[11pt]{article}
\usepackage{amsmath,amssymb,amsthm}
\usepackage[latin1]{inputenc}
\usepackage[english]{babel}
\usepackage[normalem]{ulem}
\usepackage{xcolor}
\usepackage{indentfirst} 
\usepackage{cite}

\newtheorem{thm}{Theorem}[section]

\newtheorem{cor}{Corollary}[section]
\newtheorem{lem}{Lemma}[section]
\theoremstyle{remark}
\newtheorem{rmk}{Remark}[section]
\newtheorem{exmp}{Example}[section]
\numberwithin{equation}{section}

\let \al=\alpha
\let \be=\beta
\let \var=\varphi
\let \vare=\varepsilon

\let \de=\delta
\let \th=\theta
\let \la=\lambda

\let \ga=\gamma
\let \Ga=\Gamma
\let \p=\partial
\let \q=\quad

\let \med=\medskip
\let \smal=\smallskip
\let \dps=\displaystyle
\let \ul=\underline

\let \ul=\underline
\let \ol=\overline

\newcommand{\R}{\mathbb{R}}
\newcommand{\N}{\mathbb{N}}

\newcommand{\Z}{\mathbb{Z}}

\DeclareMathOperator{\diag}{diag}

\begin{document}


\begin{center}
\textbf{\Large{Positive periodic solutions for systems of impulsive  delay differential equations }}
	\end{center}

\begin{center}
	Teresa Faria\footnote{Corresponding author. E-mail: teresa.faria@fc.ul.pt}\\

	Departamento de Matem\'{a}tica and CMAFCIO, Faculdade de Ci\^{e}ncias,\\
	Universidade de Lisboa, Campo Grande, 1749-016 Lisboa, Portugal\\

\end{center}

\begin{center} Rub\'en Figueroa \\
 Departamento de Estat\'\i stica, An\'alise Matem\'atica e Optimizaci\'on,\\
 Facultade de Matem\'aticas, Universidade de Santiago de Compostela, \\15782 Santiago de Compostela, Spain\\\end{center}

\

	
	\begin{abstract}
A  class of periodic  differential $n$-dimensional systems with patch structure with (possibly infinite) delay and nonlinear impulses is considered. These systems  incorporate very general nonlinearities and impulses whose signs may vary.  Criteria for
 the existence of at least one positive periodic solution are presented,  extending and improving previous ones established for the scalar case. Applications to systems  inspired in mathematical biology models, such as impulsive hematopoiesis and Nicholson-type  systems, are also included. \end{abstract}
 
 {\it Keywords}:  delay differential equations, impulses,  positive periodic solutions, Krasnoselskii's  fixed point theorem, Nicholson systems.

{\it 2010 Mathematics Subject Classification}: 34K13, 34K45, 92D25

	
	\medbreak
	
	\section{Introduction}
	\setcounter{equation}{0}
	
	Recently, differential equations with delays and impulses have  been proposed as models in popu\-lation dynamics, artificial neural networks, disease systems, chemical processes and in a  number of other scientific settings. They often  lead to  very realistic models for
evolutionary systems which go through sudden changes, caused by either natural phenomena,  drug administration or other artificial inputs.  Due to the real world interpretation of such equations,  in many contexts only positive solutions are of interest. In the case of   periodic  models, without and with impulses,   whether there exists any   positive periodic solution is   a prime question in applications.

	This paper is concerned with  a class of impulsive delay differential systems  written in abstract form as
	\begin{equation}\label{sys}
\left\{
\begin{array}{ll}
x_i'(t)=-d_i(t) x_i(t) + \displaystyle{\sum_{j=1,j \neq i}^n a_{ij}(t) x_j(t) +g_i(t,x_{it})} \ {\rm for}\ t\ne t_k,  \\
\Delta(x_i(t_k)):=x_i(t_k^+)-x_i(t_k)=I_{ik} (x_i(t_k)), \ k \in \Z,\q \q  i=1,\ldots,n,
\end{array}
\right. 
\end{equation}
where $(t_k)_{k \in \Z}\subset \R$  is a strictly increasing sequence,  the functions
$d_i,a_{ij}$ and $g_i$  
are continuous, nonnegative and periodic in $t\in \R$ (with a common period $\omega>0$),  
 and, as usual,  $x_t=(x_{1t},\dots,x_{nt})$ denotes 
the past history segment of the solution given by $x_t(s)=x(t+s)$ for $s\in [-\tau,0]$, where $\tau$ is the maximum time-delay. 
  The consideration of equations with infinite delay, in which case $x_t(s)=x(t+s)$ for $s\le 0$, is also possible. 
The impulses are supposed to occur with periodicity $\omega$ and satisfy some additional conditions introduced in  the next section. Clearly, 
 appropriate phase spaces and conditions on the impulses have to be chosen for delay differential equations (DDEs) with impulses, so that
the existence of solutions for the usual  initial value problems is ensured \cite{HMN,Ouahab,SP}.

Eq.~\eqref{sys} may refer to  growth models of
one or multiple   populations,  distributed over $n$ classes or patches with migration of the populations among them. For each $i$,
$x_i(t)$ is the density of the population  in  class $i$,
$a_{ij}(t)$ ($j\ne i$) are the migration coefficients from class $j$ to class $i$,
$d_i(t)$ the coefficient of instantaneous loss for class $i$ (which includes the death rate and the emigration rates for the population  leaving class $i$), and $g_i$ is the so called birth or production function.
Since \eqref{sys}  is considered within the framework of mathematical  biology or other natural sciences,  only  positive (or non-negative) solutions of \eqref{sys} are meaningful. Note that \eqref{sys} encompasses some relevant models, such as Nicholson or Mackey-Glass-type systems with patch structure and impulses. 

  Our main goal is to give sufficient conditions for the existence of at least one positive $\omega$-periodic solution  to \eqref{sys}, extending to $n$-dimensional systems   previous results established  in \cite{BF,FO19} for very broad classes of scalar impulsive DDEs. The method used here relies on  Krasnoselskii's fixed point theorem in cones, which is applied to a convenient and original operator constructed here, whose fixed points are precisely the  $\omega$-periodic solutions of \eqref{sys}.

Among other techniques, several  fixed point theorems
 have been extensively used to derive the existence of solutions, as well as the existence of periodic or almost periodic solutions, both for scalar and multidimensional DDEs.  For periodic {\it scalar} DDEs, this has been the subject of many researches, see e.g.~\cite{Chen,TangZou09,WJX} and  \cite{BF,FO19,Li_et.al,LiuTakeuchi,Yan07}, respectively for models without
   and with impulses,  and references therein.
   We remark that, even in the scalar impulsive case,  many authors restrict their analysis to DDEs with discrete delays and linear  impulses. 
  In the setting of  nonimpulsive {\it systems} of DDEs,
  Li \cite{Li} employed
 the continuation theorem of coincidence degree to show that a positive periodic solution must exist for a family of periodic competitive $n$-dimensional  Lotka-Volterra systems with distributed delay, and
later Krasnoselskii's fixed point theorem was applied  in \cite{TangZou}  and \cite{BCZ} to some classes of Lotka-Volterra systems with  discrete  delays. Recently, degree techniques were also used in \cite{AB} to investigate the existence of a nontrivial periodic solution to systems with a single discrete delay $\tau>0$ in the general form $x'(t)=f(t,x(t),x(t-\tau))$, with $f$ non-negative, continuous and periodic in $t$. 
We also refer to \cite{DingFu,Faria17,HWH,Troib,Wang+19} for  the treatment of periodic or almost periodic  multidimensional Nicholson systems.

The literature is however practically nonexistent
in what concerns the use of fixed point methods to address the existence of a positive periodic solution for  {\it impulsive systems} of DDEs, the paper of Zhang et al. \cite{ZHW} (on a planar impulsive Nicholson system) being an exception.   As far as the authors know, the new methodology  proposed here for the first time allows   handling   very broad classes of impulsive systems of DDEs, with very mild constraints on the impulses. 

	The  organization of this paper is  now  described.  Section 2 is a section of preliminaries,   where the main hypotheses for \eqref{sys} are introduced,  a suitable operator $\Phi$ on a cone $K$  is defined  and its major properties  deduced. 
Section 3 contains the main results of the paper, which establish easily verifiable sufficient conditions  for the existence of positive fixed points of $\Phi$, i.e.,
  $\omega$-periodic solutions to \eqref{sys}. 
A version of the Krasnoselskii  theorem in \cite{AgMeeRegan} is used, both in its  compressive and expansive forms.
Moreover,  as simple consequences of the main results, criteria  based on either a pointwise or an average comparison (for $t\in [0,\omega]$) of the coefficients in \eqref{sys}  are derived.
 In Section 4,  we analyse some families of systems with bounded linearities.
  Section 5 presents some selected examples inspired in mathematical biology models, within the framework of impulsive hematopoiesis and Nicholson-type  systems. A short section of conclusions ends the paper.
   	
	\section{Preliminaries}
	\setcounter{equation}{0}
	
	
%
	
We first set some notation.	
	For a compact interval $I=[\alpha,\beta]$  ($\alpha<\beta$) and $n\in\N$,  consider  the space  of the  piecewise continuous functions  on $I$ which are left-continuous in $(\alpha,\beta]$,
\begin{equation*}
\begin{split}PC(I,\mathbb{R}^n):=\{\varphi:I\to\R^n\, |\ & \varphi\ {\rm is\ continuous\ except\ for\ a\ finite\ number \ of\ points}\\ &{\rm for\ which\ there\ are}
\ \varphi(s^-)=\varphi(s),\, \varphi(s^+)\},
\end{split}
\end{equation*}  with the norm $\|\var\|_\infty=\max_{t\in I}|\var(t)|$, for some fixed norm $|\cdot |$ in $\R^n$.  For $\omega >0$, denote $PC_\omega(\R,\R^n)=\{x:\R\to \R^n \, |\,  x\ {\rm is}\  \omega$-periodic$, x_{|_{[0,\omega]}}\in PC([0,\omega],\R^n)\}$, with the supremum norm in $[0,\omega]$.
We write $\R^+=[0,\infty), \R^-=(-\infty,0]$. For $v\in\R^n$,  $v\ge 0$ stands for $v\in (\R^+)^n$ and $v>0$ for $v\in (0,\infty)^n$; for a function $x:I\to\R^n$, $x\ge 0$, $x>0$ stand for $x(t)\ge 0$, $x(t)>0$ for all $t\in I$, respectively.


  Consider the cone of non-negative elements
 $PC_\omega^+(\R,\R^n)=\{x\in PC_\omega(\R,\R^n):x\ge 0\}$.
 For spaces of continuous (rather than piecewise continuous) functions,  the similar notations $ C_\omega(\R,\R^n)$ and $ C_\omega^+(\R,\R^n)$ will be used.
When $n=1$, we  also write $C_\omega(\R)=C_\omega(\R,\R), C_\omega^+(\R)=C_\omega^+(\R,\R)$ and $PC_\omega(\R)=PC_\omega(\R,\R), PC_\omega^+(\R)=PC_\omega^+(\R,\R)$. 
 Here, $\R^n$ is also  seen as the set of constant functions (defined on an interval $I$ or $\R$).  

For  a fixed finite time-delay $\tau>0$, set $PC:=PC([-\tau,0],\R^n)$ as the phase space. 
Consider  the $n$-dimensional  DDE with impulses \eqref{sys}, 
for which  initial conditions have the form
$x_\sigma=\var$ for  $(\sigma, \var)\in \R\times PC$.
The following hypotheses on \eqref{sys} will be assumed:
\begin{itemize}
\item[(H1)] The functions $I_{ik}:\R^+\to\mathbb{R}$ are continuous and    there is a positive integer $p$ such that
	$0\le t_1<\cdots <t_p< \omega$  (for some $\omega>0$) and 
	$t_{k+p}=t_k+\omega,\ I_{i,k+p}=I_{ik},\ k\in\Z, i=1,\dots,n;$
	\item[(H2)] There exist constants $\al_{ik}>-1$ and $\eta_{ik}$ such that $\al_{ik}u\leq I_{ik}(u)\leq \eta_{ik} u$ for  $u\geq 0$ and   there are the limits $\dps\lim_{u\to 0^+}\frac{u}{u+I_{ik}(u)},$
	for $i=1,\dots,n,k=1,\dots,p$;
	\item[(H3)] $ \prod_{k=1}^p(1+\eta_{ik})<e^{\int_0^\omega d_i(t)\, dt},\ i=1,\dots,n$;
	\item[(H4)] (i) For $i,j=1,\dots,n$,  $d_i,a_{ij}\in C_\omega^+(\R)$ with $\int_0^\omega d_i(s)\, ds>0$, the functions $g_i:\R\times PC([-\tau,0],\R)\to \R^+$ are  continuous,  $\omega$-periodic  in $t\in\R$ and
	$$g(t,x_t):=(g_1(t,x_{1t}),\dots,g_n(t,x_{nt}))$$ is  bounded on bounded sets of $\R\times PC$; \\ (ii) moreover,  if $n>1$,   either $ \int_0^\omega a_{ij}(s)\, ds>0$ for all $i\ne j$   or
  $\int_0^\omega g_i(s,0)\, ds>0$, for each  $i=1,\dots,n$.
\end{itemize}

For  fixed $\omega>0$,  $n\in\N$ and a sequence
 $(t_k)_{k\in\Z}$ as in (H1), define the space
  \begin{equation}\label{spaceX}
 \begin{split}
 X:=X(\R^n)=\{ x:\R\to\R^n\ |&\ x\ {\rm is\ } \omega{\rm -periodic,\ continuous \ for\  all}\  t\neq t_k,\\
&{\rm and}\  x(t_k^-)=x(t_k),\ x(t_k^+)\in\R,\ {\rm for}\  k\in\Z\},
 \end{split}
 \end{equation}
 and the cone
  \begin{equation}\label{X+}
  X^+:=X(\R^n)^+=\{x\in X: x(t)\ge 0,\ t\in[0,\omega]\}.
  \end{equation}
 Hereafter,  $X$ is endowed with the norm $\|\cdot\|_\infty$ (where  the maximum norm in taken in $\R^n$), simply denoted by $\|\cdot\|$, and with the partial order  induced by the cone $X^+$.

\begin{rmk}\label{rmk2.1}
In the case of infinite delay, as phase space we may take any admissible Banach space $({\cal B},\|\cdot\|_{\cal B})$ (in the sense of Hale and Kato definition \cite{HaleKato}) of functions from $\R^-$ to $\R^n,$ such that ${\cal B}$ contains the space $ PC_\omega(\R^-,\R^n)$ of piecewise continuous, $\omega$-periodic functions $x:\R^-\to\R^n$,
 and such that the norms $\|\cdot\|$ and $\|\cdot\|_{\cal B}$ are equivalent in $ PC_\omega(\R^-,\R^n)$. 
 See \cite{BF,HMN} for details. To simplify the exposition, below we only consider systems with finite delay, although straightforward adjustments can be effected to deal with the infinite delay case.
\end{rmk}

We remark that (H2) implies that $I_{ik}(u)>-u$ for $u>0$, hence a positive solution of \eqref{sys} will remain positive after suffering an impulse at each instant $t_k$. If $n>1$, without loss of generality we take $a_{ii}\equiv 0$ for $i=1,\dots,n$. 
For $d_i\in C_\omega^+(\R)$, the requirement
$\int_0^\omega d_i(s)\, ds>0$ guarantees that $d_i$ is not identically zero. Similarly, with  $a_{ij},g_i(\cdot,0)\in C_\omega^+(\R)$, $a_{ij}\not\equiv 0$ if
$\int_0^\omega a_{ij}(s)\, ds>0\ (j\ne i)$ and $g_i(\cdot,0)\not\equiv 0$ if $\int_0^\omega g_i(s,0)\, ds>0$. The role of assumption (H4)(ii) is to preclude the existence of periodic solutions with one component positive but with others that may vanish. See additional comments on Remark \ref{rmk2.2}.

In order to simplify the exposition,
 for $i=1,\dots,n, k=1,\dots,p, t\in\R$, consider the following auxiliary functions: 
\begin{align}\label{B&Js}
&D_i(t)=\int_0^t d_i(s)\, ds,\q J_{ik}(u)=\left\{
\begin{array}{ll}
\dfrac{u}{u+I_{ik}(u)}\ ,\ u>0,  \\
\dps \lim_{u\to 0^+}\frac{u}{u+I_{ik}(u)}\ ,\ u=0
\end{array}
\right.
\\
& B_i(t;x_i)=\displaystyle\prod_{k:t_k\in[0,t)}J_{ik}(x_i(t_k))\q {\rm and}\\
&
\tilde{B_i}(s,t;x_i)=\frac{B_i(s;x_i)}{B_i(t;x_i)}=\displaystyle\prod_{k:t_k\in[t,s)}J_{ik}(x_i(t_k)) \q {\rm for}\ 0\leq t\leq s\leq t+\omega, x\in X^+,
\end{align}
and the autonomous quantities
\begin{align}\label{D&Gas}
&D_i(\omega)=\int_0^\omega d_i(s)\, ds,\q B_i(\omega;x_i)=\prod_{k=1}^p J_{ik}(x_i(t_k)),\\
& \Ga_i(x_i)=\Big (B_i(\omega;x_i)e^{D_i(\omega)}-1\Big)^{-1}\q {\rm for}\q i=1,\dots,n,   x\in X^+. \end{align}
We adopt the usual convention that a product is equal to one when the number of factors is zero.

For systems without impulses, clearly $J_{ik}(u)\equiv 1, B_i(t;x_i)\equiv 1,\Ga_i(x_i)\equiv \big (e^{D_i(\omega)}-1\big)^{-1}$.
On the other hand, when all impulses are linear, i.e., $I_{ik}(u)=\eta_{ik}u$ for some constants $\eta_{ik}>-1$ with 
 (H1),(H3) fulfilled,   $J_{ik}$ are also constants, $J_{ik}\equiv (1+\eta_{ik})^{-1}$, thus the  functions $B_i$ and $\tilde B_i$ do not depend on $x$.



We state some properties of these auxiliary functions, whose validity is easily verified by adapting the arguments for the scalar version of \eqref{sys}; the reader is referred to \cite{BF}    to complete a proof of the properties below.

\begin{lem}\label{lem2.1} Assume (H1)--(H4). For $i=1,\dots,n, k\in \Z, x=(x_1,\dots,x_n)\in X^+$:

(i) $J_{ik}:\R^+\to (0,\infty), \Ga_i: X^+(\R)\to (0,\infty)$ are continuous and satisfy
\begin{align*}
(1+\eta_{ik})^{-1}&\le J_{ik}(u)\le (1+\al_{ik})^{-1},\q u\ge 0,\\
\ul{\Ga_i}&\le \Ga_i(x_i)\le \ol{\Ga_i}\, ,
\end{align*} where
$\ul{\Ga_i}:=\Big (\prod_{k=1}^p (1+\al_{ik})^{-1}e^{D_i(\omega)}-1\Big)^{-1},\ol{\Ga_i}:=\Big (\prod_{k=1}^p (1+\eta_{ik})^{-1}e^{D_i(\omega)}-1\Big)^{-1}$;

(iii) $B_i(t+\omega;x_i)=B_i(t;x_i)B(\omega;x_i)$ for $t\in\R$;

(iv) $B_i(t_k+\vare;x_i)=B_i(t_k;x_i)J_{ik}(x_i(t_k))^{-1}$  for $0<\vare<\min_{1\le k\le p} (t_{k+1}-t_k)$;

(v) $\tilde{B_i}(s,t;x_i)$ are bounded functions on $D\times X^+(\R)$, where $D=\{(s,t)\in\R^2: t\le s\le t+\omega\}$, with
$$\ul{B_i}\le \tilde{B_i}(s,t;x_i)\le \ol{B_i}\q {\rm for}\q (s,t,x_i)\in D\times X^+(\R),$$ where
 $\underline{B_i}:=\min \big \{ \prod_{k=j}^{j+l-1} (1+\eta_{ik})^{-1}: j=1,\dots ,p, l=0,\dots, p\big\}$, $\overline{B_i}:=\max \big \{ \prod_{k=j}^{j+l-1} (1+\al_{ik})^{-1}: j=1,\dots ,p, l=0,\dots, p\big\} $;

(vi) $\tilde{B_i}(s+\omega,t+\omega;x_i)=\tilde{B_i}(s,t;x_i)$ for $(s,t,x_i)\in D\times X^+(\R)$;

(vii) if $x(t)=(x_1(t),\dots,x_n(t))$ is a solution of \eqref{sys}, the function $y(t)=(y_1(t),\dots,y_n(t))$, where 
\begin{equation}\label{contSol}y_i(t)=B_i(t;x_i)x_i(t),\q i=1,\dots,n,
\end{equation} is continuous.
\end{lem}

From Lemma \ref{lem2.1}, we obtain: 
\begin{lem}\label{lem2.2} Assume (H1)--(H4). The operator $\Phi:X^+\to X^+$ given by
\begin{equation}\label{Phi}
\begin{split}
 \Phi &=(\Phi_1,\dots, \Phi_n),\\
(\Phi_i x)(t)&=\Ga_i(x_i)\int_t^{t+\omega}\tilde{B_i}(s,t;x_i)e^{\int_t^sd_i(r)\, dr}\left (\sum_{j\ne i}a_{ij}(s)x_j(s)+g_i(s,x_{is})\right)ds,\, \, t\in\R,
\end{split}
\end{equation}
is well defined. Moreover,  $x$ is a nonnegative $\omega$-periodic solution of \eqref{sys} if and only if $x$ is a fixed point of $\Phi$. 
%
%
\end{lem}

\begin{proof} We argue along the major lines in \cite{BF,FO19}. Let $x=(x_1,\dots,x_n)\in X^+$. Clearly,  $\Phi x\geq 0$,  $t\mapsto(\Phi x)(t)$ is continuous for  $t\ne t_k$ and left-continuous on $t_k$, $k=1,\dots,p$. The properties in Lemma \ref{lem2.1}  show that $(\Phi x)(t)$ is $\omega$-periodic and that
\begin{equation}\label{impulsesPhi}(\Phi_i x)(t_k^+)=\lim_{\vare\to 0^+} (\Phi_i x)(t_k+\vare)=J_{ik}(x_i(t_k))^{-1}(\Phi_i x)(t_k)
\end{equation}
for all $i\in\{1,\dots,n\}$ and $k\in\Z$, thus $\Phi(X^+)\subset X^+$.


Take $x=(x_1,\dots,x_n)\in X^+$ and suppose that $x$ is a  solution of \eqref{sys}. For the continuous function $y(t)$ with components as in \eqref{contSol} and $t\ne t_k \, (k\in\Z)$, we have
$$\Big (y_i'(t)+d_i(t)y_i(t)\Big )e^{D_i(t)}=B_i(t;x_i) e^{D_i(t)}\bigg (\sum_{j\ne i} a_{ij}(t)x_j(t)+g_i(t,x_{it})\bigg).$$
Since $x_i(t)$ is $\omega$-periodic and $y_i(t)$ is continuous, integration over intervals $[t,t+\omega]$, the properties in Lemma \ref{lem2.1} and computations as in \cite{BF} lead to
$$x_i(t)B_i(t;x_i)e^{D_i(t)}\Big (B_i(\omega;x_i)e^{D_i(\omega)}-1\Big)=\int_t^{t+\omega} B_i(s;x_i)e^{D_i(s)}\bigg (\sum_{j\ne i}a_{ij}(s)x_j(s)+g_i(s,x_{is})\bigg)ds,$$
thus $x_i(t)=(\Phi_i x)(t)$ for all $i$ and $t$, and $x$ is a fixed point of $\Phi$.

Conversely if $x\in X^+$ is a fixed point of $\Phi$, for $t\ne t_k$ differentiation of $x_i(t) \, (1\le i\le n)$ gives
\begin{equation*}
\begin{split}
x'_i(t)&=(\Phi_i x)'(t)\\
&=-d_i(t)(\Phi_i x)(t)+\Ga_i(x_i)\Big (B_i(\omega;x_i)e^{D_i(\omega)}-1\Big) \Big (\sum_{j\ne i}a_{ij}(t)x_j(t)+g_i(t,x_{it})\Big )\\
&=-d_i(t) x_i(t)+\sum_{j\ne i}a_{ij}(t)x_j(t)+g_i(t,x_{it}).
\end{split}
\end{equation*}
On the other hand, for $t=t_k$,  from \eqref{impulsesPhi} we get 
\begin{equation*}
\begin{split}\Delta x_i(t_k)&=(\Phi_i x)(t_k^+)-x_i(t_k)\\
&=J_{ik}(x_i(t_k))^{-1}(\Phi_i x)(t_k)-x_i(t_k)\\
&=[J_{ik}(x_i(t_k))^{-1}-1]x_i(t_k)=I_{ik}(x_i(t_k)).
\end{split}
\end{equation*}
Therefore, $x$ is a solution of \eqref{sys}.
\end{proof}

For any  $\sigma=(\sigma_1,\dots,\sigma_n)\in(0,1)^n$,  consider a new cone $K(\sigma)$ in $X$ given by
\begin{equation}\label{K}
K({\sigma}):=\{x\in X^+:x_i(t)\geq\sigma_i \|x_i\|, t\in [0,\omega], i=1,\dots,n\}.
\end{equation}
If  $\sigma\in(0,1)^n$ is fixed, we denote $K({\sigma})$ simply by $K$
and $$K_0=K_0({\sigma}):=\{x\in K:x_i(t)>0, t\in [0,\omega],i=1,\dots,n\}.$$
The aim is to prove the existence of  a fixed point of $\Phi$ in $K_0$, so that a positive $\omega$-periodic solution of \eqref{sys} exists.
For this, a  Krasnoselskii  fixed point theorem in the version in \cite[Theorems   7.3 and 7.6]{AgMeeRegan}, which includes both the compressive and expansive forms, will be used.

%
%
%
%


\begin{thm}\label{thmKras}\cite{AgMeeRegan} Let $K$ be a closed cone in a Banach space, $r,R\in\R^+$ with $r\ne R$, $r_0=\min\{r,R\} , R_0= \max\{r,R\}$ and  
 $K_{R_0}:=\{x\in K:\|x\| \leq R_0\}$. Assume that
 $T:K_{R_0} \longrightarrow K$ is a completely continuous operator such that
\begin{enumerate}
\item  $Tx\ne \la x$ for all $x \in K$ with $\|x\|=R$ and all $\la>1$;
\item There exists $\psi \in K\setminus\{0\}$ such that $x \neq Tx + \lambda \psi$ for all $x \in K$ with $\|x\|=r$ and all $\lambda >0$.
\end{enumerate}
Then $T$ has a fixed point  in $K_{r_0,R_0}:=\{x\in K:r_0\leq\|x\| \leq R_0\}$.
\end{thm}

For $g$ as in (H4), we also define 
\begin{equation}\label{G}
G(t,x)=g(t,x_t)\q {\rm for}\q t\in\R, x\in X^+.\end{equation}
To derive the compactness of the operator $\Phi$,
an additional hypothesis on $g$ is assumed:
\begin{itemize}
	\item[(H5)] The function $t\mapsto G(t,x)$ is {\it uniformly equicontinuous}  for $t\in [0,\omega]$ on  bounded sets of $K$, in the sense that for any $A\subset K$ bounded and $\vare>0$, there is $\de>0$ such that $\max_{t\in [0,\omega]}|G(t,x)-G(t,y)|<\vare$ for all  $x,y\in A$ with $\|x-y\|<\de$.
\end{itemize}

\begin{lem}\label{lem2.3} Assume (H1)--(H4), consider $\sigma=(\sigma_1,\dots,\sigma_n)$ with $0<\sigma_i\le \ul{B_i}\ol{B_i}^{-1}e^{-D_i(\omega)}$ for
 $i=1,\dots,n$, and $K=K(\sigma)$. Then:\\
 (i) $\Phi(K)\subset K$.\\
 (ii)  If $x\in K\setminus \{0\}$ is a fixed point of   $\Phi$, then $x$ is a positive $\omega$-periodic solution of \eqref{sys}. \\
 (iii) If in addition (H5) holds, $\Phi$ is completely continuous.
\end{lem}

\begin{proof} 
(i) From Lemma \ref{lem2.2}, $\Phi(X^+)\subset X^+$.
Now, take $x=(x_1,\dots,x_n)\in K$. 
For $\sigma$ chosen as above,
$$(\Phi_i x)(t)\le \Ga_i(x_i)e^{D_i(\omega)}\ol{B_i}\int_0^{\omega}\bigg (\sum_{j\ne i}a_{ij}(s)x_j(s)+g_i(s,x_{is})\bigg)ds$$
and
\begin{equation}\label{2.10}(\Phi_i x)(t)\ge \Ga_i(x_i)\ul{B_i}\int_0^{\omega}\bigg (\sum_{j\ne i}a_{ij}(s)x_j(s)+g_i(s,x_{is})\bigg)ds,\end{equation}
leading to $(\Phi_i x)(t)\ge \sigma_i \|\Phi_i x\|$
for all $i$ and all $t$. Thus, $\Phi(K)\subset K$.
\smal


(ii) 
If $x\in K, x\ne 0$ and $x=\Phi x$, from Lemma \ref{lem2.2} $x(t)$ is a nontrivial $\omega$-periodic solution of \eqref{sys}. If $n=1$, the definition of $K$ implies that $K\setminus \{0\}=K_0$. If $n>1$, there is $i^*\in \{1,\dots,n\}$ such that $x_{i^*}(t)\ge \sigma_{i^*}\|x_{i^*}\|=\sigma_{i^*}\|x\|>0, t\in [0,\omega]$.  For $i\ne i^*$, either $x_i(t)\ge \|x_i\|>0$ for all $t$, or  $x_i=(\Phi_i x)\equiv 0$; in the latter case, from (H4)(ii) it then  follows that 
$$0=(\Phi_i x)(t)\ge \Ga_i(x_i)\ul{B_i}\int_0^{\omega}\Big (a_{ii^*}(s)x_{i^*}(s)+g_i(s,0)\Big )\,ds  >0,$$
for $ t\in [0,\omega]$, which is not possible. Hence $x_i(t)\ge \sigma_i\|x_i\|>0$.
Therefore, all  the components of $x$ are strictly positive on $[0,\omega]$, i.e., $x\in K_0$.
\smal


(iii)  
The proof follows by a straightforward adaptation of the arguments for the scalar case in \cite{BF}, replacing a scalar function $G(t,x)$ by the functions $H_i(t,x)=\sum_{j\ne i} a_{ij}(t)x_j(t)+G_i(t,x_i)\, (1\le i\le n)$, for $G$ as in \eqref{G}. Clearly, the function $H=(H_1,\dots,H_n)$ also satisfies (H5),
hence the proof  in \cite{BF} applies to the present situation.
\end{proof}

\begin{rmk}\label{rmk2.2} As previously mentioned, the role of  (H4)(ii) is  to preclude the existence of nontrivial fixed points  of $\Phi$ with one or more coordinates equal to zero. In this way, it can be replaced by any other assumption with the same outcome. We point out  that some authors \cite{BCZ,TangZou}   have imposed hypotheses and 
employed Krasnoselskii's techniques to some classes  of periodic systems of DDEs (without impulses), which  however  seem to only guarantee  that a  {\it nontrivial}, rather than positive, periodic solution must exist. 
\end{rmk}

\section{Main results}
\setcounter{equation}{0}

In this section, general criteria for the existence of positive periodic solutions of \eqref{sys} are given.

To use the  compressive form of Krasnoselskii's  fixed point theorem,  we impose the assumption:
\begin{itemize}
\item[(H6)]  There are constants $r_0,R_0$ with $0<r_0< R_0 $ and  functions $b_{1i},b_{2i}\in C_\omega^+(\R)$ with  $\int_0^\omega b_{qi}(t)\, dt> 0\, (q=1,2)$, such that
	 for  $i=1,\dots,n$, $x\in K$ and $t\in [0,\omega]$ it holds: 
		\begin{equation}\label{sublinear_bound}
		\begin{split}
		&g_i(t,x_{it})\geq b_{1i}(t)u\quad \text{if}\:\:0< u\leq x_i\leq r_0,\\
		&g_i(t,x_{it})\leq b_{2i}(t)u\quad \text{if}\:\:R_0\leq x_i\leq u.
		\end{split}
		\end{equation}
		\end{itemize}
		
Bearing in mind the behaviour of the nonlinearities in \eqref{sys}  at infinity, when  (H6) holds we say that  \eqref{sys} is {\it sublinear}   -- this is the situation of Mackey-Glass and Nicholson systems, as well as other important models from mathematical biology.

\begin{thm}\label{thm3.1}  Assume (H1)--(H6) and  that, for $b_{1i},b_{2i}$ as in (H6),
 \begin{equation}\label{sublinear0}
 \begin{split}
&\ul{\Ga_i}\, \ul{B_i}\min_{t\in[0,\omega]}\int_t^{t+\omega}e^{\int _t^sd_i(r)\, dr}\Big (\sum_{j\ne i}a_{ij}(s)+b_{1i}(s)\Big )\, ds\ge 1,\\
& \ol{\Ga_i}\, \ol{B_i}\ \max_{t\in[0,\omega]}\int_t^{t+\omega}e^{\int _t^sd_i(r)\, dr}\Big (\sum_{j\ne i}a_{ij}(s)+b_{2i}(s)\Big )\, ds\le 1,\q i=1,\dots,n.
 \end{split}
		\end{equation}
 Then there exists (at least) one positive $\omega$-periodic solution  $x^*(t)$ of \eqref{sys} satisfying
$$ \min_{t\in [0,\omega]}x_i^*(t)\ge \sigma_i\max_{t\in [0,\omega]}x_i^*(t),\q i=1,\dots,n,$$
  for  $0<\sigma_i\le \ul{B_i}\ol{B_i}^{-1}e^{-D_i(\omega)}\ (1\le i\le n)$ as in Lemma \ref{lem2.3}.	
		\end{thm}
\begin{proof} Fix $r_0,R_0$ as in (H6).
Let $R\ge R_0(\min_{1\le i\le n}\sigma_i)^{-1}$ and $x\in K$ with $\|x\|=R$ . 
	Choose $i$ such that $\|x\|=\|x_i\|=R$. For such $i$, we have $x_i(t)\le R$ and $x_i(t)\ge \sigma_i \|x_i\|= \sigma_i R\ge R_0$ for $t\in [0,\omega]$, therefore, from the second inequality  in \eqref{sublinear_bound} we obtain
	$$g_i(t,x_{it})\le b_{2i}(t)R.$$
	Using the properties in Lemma \ref{lem2.1} and \eqref{sublinear0}, we have
 \begin{equation}\label{3.6}\|\Phi_i x\|\le R\, \ol{\Ga_i}\, \ol{B_i}\max_{t\in[0,\omega]}\int_t^{t+\omega} e^{\int _t^sd_i(r)\, dr}\Big [\sum_{j\ne i}a_{ij}(s)+b_{2i}(s))\Big ]\, ds\le R.\end{equation}
In particular, we conclude that $\Phi x\ne \la x$ for all $\la >1$ and $x\in K$ with $\|x\|=R$.
		
		On the other hand, take $r\le \min_{1\le i\le n}\sigma_i r_0$, $\psi\equiv {\bf 1}:=(1,\dots,1)$ and consider any $\la>0$. For  $x\in K$ with $\|x\|=r$, we claim that $x\ne \Phi x+\la \psi$. 
		
		Suppose otherwise that	 there are $\la >0, x \in K$ with $\|x\|=r$ and $x=\Phi x+\la {\bf 1}$. 
		Let $\mu:=\min_{t\in [0,\omega]}\min_{1\le i\le n}x_i(t)$. We first note that, for $t\in [0,\omega], i=1,\dots,n$, we have $0<\la\le \mu \le x_i(t)\le r\le r_0$, thus the first inequality  in \eqref{sublinear_bound} implies
		$$g_i(t,x_{it})\ge b_{1i}(t)\mu,$$
which, together with the first constraint in  \eqref{sublinear0}, yields for all $i=1,\dots,n$ and $t\in [0,\omega]$  that
 \begin{equation}\label{3.7}(\Phi_i x)(t)\ge \mu\, \ul{\Ga_i}\, \ul{B_i}\min_{t\in[0,\omega]}\int_t^{t+\omega} e^{\int _t^sd_i(r)\, dr}\Big [\sum_{j\ne i}a_{ij}(s)+b_{1i}(s)\Big ]\, ds\ge\mu.\end{equation}
Next, choose 
		$t^*\in [0,\omega]$ and $i^*\in \{1,\dots,n\}$ such that $x_{i^*}(t^*)< \mu+\la$. We obtain
\begin{equation*}
\begin{split}
\mu&> x_{i^*}(t^*)-\la=(\Phi_{i^*} x)(t^*)\ge \mu,
 \end{split}
\end{equation*}	
which is not possible.	The claim is proven, thus Theorem \ref{thmKras} provides the existence of a fixed point $x^*$  for $\Phi$ in $K_{r,R}=\{x\in K:r\leq\|x\| \leq R\}$. From Lemma \ref{lem2.3}(ii), this fixed point is a  positive $\omega$-periodic solution of \eqref{sys}.
%
	\end{proof}
		A scaling of the variables allows us to obtain an algebraic variant  of Theorem \ref{thm3.1}, which turns out to be very useful.

\begin{thm}\label{thm3.2} Assume (H1)--(H6) and  that there is $v=(v_1,\dots,v_n)>0$ such that, for $b_{1i},b_{2i}$ as in (H6),
 \begin{equation}\label{sublinear}
 \begin{split}
c_i^0(v):=&\ul{\Ga_i}\, \ul{B_i}\min_{t\in[0,\omega]}\int_t^{t+\omega}e^{\int _t^sd_i(r)\, dr}\Big (\sum_{j\ne i}v_i^{-1}v_ja_{ij}(s)+b_{1i}(s)\Big )\, ds\ge1,\\
C_i^\infty(v):=& \ol{\Ga_i}\, \ol{B_i}\ \max_{t\in[0,\omega]}\int_t^{t+\omega}e^{\int _t^sd_i(r)\, dr}\Big (\sum_{j\ne i}v_i^{-1}v_ja_{ij}(s)+b_{2i}(s)\Big )\, ds\le1,\q i=1,\dots,n.
 \end{split}
 \end{equation}
Then there exists (at least) one positive $\omega$-periodic solution  $x^*(t)$ of \eqref{sys}.	
		\end{thm}
		
	\begin{proof} 
Effecting the change of variables $\bar x_i=v_i^{-1}x_i\ (1\le i\le n)$ and dropping the bars for simplicity, system \eqref{sys} becomes
\begin{equation}\label{SysX}
\left\{
\begin{array}{ll}
x_i'(t)=-d_i(t) x_i(t) + \displaystyle{\sum_{j \neq i} v_i^{-1}v_ja_{ij}(t) x_j(t) +\tilde g_i(t,x_{it})} \ {\rm for}\ t\ne t_k, \\
x_i(t_k^+)-x_i(t_k)=\tilde I_{ik} (x_i(t_k)), \ k \in \Z, 
\end{array}
\right.  i=1,\ldots,n,
\end{equation}
where $\tilde g_i(t,u)=v_i^{-1}g_i(t,v_iu),\tilde I_{ik}(u)= v_i^{-1} I_{ik}(v_iu)$ for all $i,k$. On the one hand, the functions $\tilde I_{ik}(u)$ satisfy hypotheses (H1)--(H3) with the same constants $\al_{ik},\eta_{ik}$, and $\tilde J_{ik}(u):=\frac{u}{u+\tilde I_{ik}(u)}=J_{ik}(v_iu)\, (u>0)$, for $J_{ik}$ as in
\eqref{B&Js}. On the other hand, if the functions $g_i(t,u)$ satisfy \eqref{sublinear_bound}, then $\tilde g_i(t,u)$ satisfy  \eqref{sublinear_bound} as well. Consequently,  Theorem \ref{thm3.1} implies the result.
		\end{proof}

The superlinear case of \eqref{sys} 	is dealt in a similar way, by using the expansive form of Krasnoselskii's theorem. The proof is omitted.

\begin{thm}\label{thm3.2_super} Assume (H1)--(H5) and	
\item[(H7)] 	There are	constants $r_0,R_0$  with $0<r_0< R_0 $  and  functions $b_{1i},b_{2i}\in C_\omega^+(\R)$ with  $\int_0^\omega b_{qi}(t)\, dt> 0\, (q=1,2)$, such that	
		 for $i=1,\dots,n$,  $x\in K$ and $t\in [0,\omega]$ it holds: 	\begin{equation}\label{superlinear_bound}
		\begin{split}
		&g_i(t,x_{it})\leq b_{1i}(t)u\quad \text{if}\:\:0< x_i\le u\leq r_0,\\
		&g_i(t,x_{it})\geq b_{2i}(t)u\quad \text{if}\:\:x_i\ge u\ge R_0.
		\end{split}
		\end{equation}
If there is a vector $v=(v_1,\dots,v_n)>0$ such that
 \begin{equation}\label{superlinear}
		\begin{split}
& C_i^0(v):=\ol{\Ga_i}\, \ol{B_i}\max_{t\in[0,\omega]}\int_t^{t+\omega} e^{\int _t^sd_i(r)\, dr}\Big (\sum_{j\ne i}v_i^{-1}v_ja_{ij}(s)+b_{1i}(s)\Big )\, ds\le 1,\\
&c_i^\infty(v):=\ul{\Ga_i}\, \ul{B_i}\min_{t\in[0,\omega]}\int_t^{t+\omega} e^{\int _t^sd_i(r)\, dr}\Big (\sum_{j\ne i}v_i^{-1}v_ja_{ij}(s)+b_{2i}(s)\Big )\, ds\ge1,\q i=1,\dots,n.
 \end{split}
		\end{equation}
		then \eqref{sys} has at least one positive $\omega$-periodic solution.
		\end{thm}
		
\begin{rmk}\label{rmk3.1}		
Additionally, variations of Krasnoselskii's theorem, where the compressive and expansive forms are combined, can lead to the existence of more than one positive $\omega$-periodic solution to \eqref{sys}, as in e.g.  \cite{Li_et.al} and several other works.
\end{rmk}
\begin{rmk}\label{rmk3.2'}	Due to the  version of Krasnoselskii's theorem  given in Theorem  \ref{thmKras}, we stress that the equalities to 1 are allowed in all conditions
\eqref{sublinear},\eqref{superlinear}. Hence, even  for the scalar case, the above theorems  lead to improvements of some criteria in \cite{BF}, where the strict inequalities were required.
\end{rmk}	
\begin{rmk}\label{rmk3.2} Under (H1)-(H6), it is clear that the sufficient conditions  expressed by \eqref{sublinear0} or \eqref{sublinear} are not optimal, since one can use sharper estimates for $\Gamma_i(u), \tilde B_i(s,t;u)$, for both  $u$ in the vicinity of 0 and $\infty$. In fact,
at $\infty$ define 
$\dps B_i(t;\infty^+):=\limsup_{u\to\infty}\prod_{k:t_k\in[0,t)}J_{ik}(u), B_i(t;\infty^-):=\liminf_{u\to\infty}\prod_{k:t_k\in[0,t)}J_{ik}(u)$
and
$\tilde{B_i}(s,t;\infty^\pm):=B_i(s;\infty^\pm)B_i(t;\infty^\mp)^{-1}, \Ga_i(\infty^+):=\Big (B_i(\omega;\infty^-)e^{D_i(\omega)}-1\Big)^{-1}.$
It is clear that Theorem \ref{thm3.2} still holds if conditions \eqref{sublinear}  are replaced by
 \begin{equation*}\label{sublinear00}
 \begin{split}
		 &\Ga_i(0)\min_{t\in[0,\omega]}\int_t^{t+\omega} \tilde{B_i}(s,t;0)e^{\int _t^sd_i(r)\, dr}\Big (\sum_{j\ne i}v_i^{-1}v_ja_{ij}(s)+b_{1i}(s)\Big )\, ds\ge1,\\
		 & \Ga_i(\infty^+)\, \max_{t\in[0,\omega]}\int_t^{t+\omega} \tilde{B_i}(s,t;\infty^+)e^{\int _t^sd_i(r)\, dr}\Big (\sum_{j\ne i}v_i^{-1}v_ja_{ij}(s)+b_{2i}(s)\Big )\, ds\le1,\ i=1,\dots,n.
 \end{split}
		\end{equation*}
Though not as sharp as the ones above, the estimates in  \eqref{sublinear} are  much easier to verify in practice. A similar improvement can be effected to the statement for the superlinear case in Theorem \ref{thm3.2_super}.
\end{rmk}

\med

For \eqref{sys}, define  the  $n\times n$ matrices of  functions in $C_\omega^+(\R)$ given by
\begin{equation}\label{D&A}
D(t)=\diag \,  (d_1(t),\dots,d_n(t)),\q  A(t)=\big [a_{ij}(t)\big ],\end{equation}
where as before $a_{ii}(t)\equiv 0$ for all $i$. Criteria for the existence of a positive periodic solution  
 involving either  a pointwise comparison of the functions $d_i(t),a_{ij}(t)$ and $b_{qi}(t)$
  or their integral averages  are  useful in applications. Both these approaches are considered in the next theorem.


\begin{thm}\label{thm3.3} Assume  (H1)--(H6). For $b_{1i}(t),b_{2i}(t)$ as in (H6), define
\begin{equation}\label{BB}
B_1(t)=\diag \, (b_{11}(t),\dots,b_{1n}(t)),\q B_2(t)=\diag \, (b_{21}(t),\dots,b_{2n}(t)),
\end{equation}
for $t\in\R$. With   $D(t),A(t)$ as in \eqref{D&A} and some  vector $v>0$, assume one of conditions:  \begin{itemize}  \item[(a)] either
\begin{equation}\label{H8}
 M_2 \Big [B_2(t)+A(t)\Big ]v\le D(t)v\le M_1 \Big [B_1(t)+A(t)\Big ]v,
\end{equation}
 for 
\begin{equation}\label{mm} 
\begin{split}
M_1=\diag (m_{11},\dots,m_{1n}),\ &M_2=\diag (m_{21},\dots,m_{2n}),\\
 m_{1i}:=\ul{\Ga_i}\, \ul{B_i}(e^{D_i(\omega)}-1),\ &m_{2i}:= \ol{\Ga_i}\, \ol{B_i}(e^{D_i(\omega)}-1),\q i=1,\dots,n;
 \end{split}
\end{equation}
  \item[(b)] or  \begin{equation}\label{H9}
  \int_0^\omega N_2 \Big [B_2(t)+A(t)\Big ]v\, dt\le v\le \int_0^\omega N_1\Big [B_1(t)+A(t)\Big ]v\, dt,\\
\end{equation}  for 
  \begin{equation}\label{nn}
  \begin{split}
 N_1=\diag(n_{11},\dots,n_{1n}),\
  &N_2=(n_{21},\dots,n_{2n}),\\
   n_{1i}:=\ul{\Ga_i}\, \ul{B_i},\ &n_{2i}:= \ol{\Ga_i}\, \ol{B_i}e^{D_i(\omega)},\q i=1,\dots,n.
  \end{split}
  \end{equation}
\end{itemize}
Then, \eqref{sys} has (at least) one positive $\omega$-periodic solution.
\end{thm}
		
		\begin{proof}
	Consider	$ c_i^0(v),C_i^\infty(v)\ (1\le i\le n)$ defined in \eqref{sublinear}. From (a), we have 
	$$\sum_{j\ne i}v_ja_{ij}(s)+v_ib_{1i}(s)\ge m_{1i}^{-1} v_id_i(s),\q
	\sum_{j\ne i}v_ja_{ij}(s)+v_ib_{2i}(s)\le m_{2i}^{-1} v_id_i(s),$$
thus
\begin{equation*}
\begin{split}
 c_i^0(v)&\ge m_{1i}^{-1}\ul{\Ga_i}\, \ul{B_i}\min_{t\in[0,\omega]}\int_t^{t+\omega} e^{\int _t^sd_i(r)\, dr} d_i(s)\, ds =m_{1i}^{-1}
  \ul{\Ga_i}\, \ul{B_i}(e^{D_i(\omega)}-1)= 1,\\
C_i^\infty(v)&\le m_{2i}^{-1} \ol{\Ga_i}\, \ol{B_i}\max_{t\in[0,\omega]}\int_t^{t+\omega} e^{\int _t^sd_i(r)\, dr} d_i(s)\, ds = m_{2i}^{-1}\ol{\Ga_i}\, \ol{B_i}(e^{D_i(\omega)}-1)=1,
\end{split}
\end{equation*}
for $i=1,\dots,n$. If (b) is satisfied, for all $i$ we obtain
\begin{equation*}
\begin{split}
 c_i^0(v)&>v_i^{-1}\ul{\Ga_i}\, \ul{B_i}\int_0^\omega \Big (\sum_{j\ne i}v_ja_{ij}(s)+v_ib_{1i}(s)\Big )\, ds\ge v_i^{-1}\ul{\Ga_i}\, \ul{B_i}v_in_{1i}^{-1}= 1,\\
C_i^\infty(v)&< v_i^{-1} \ol{\Ga_i}\, \ol{B_i} e^{D_i(\omega)} \int_0^\omega \Big (\sum_{j\ne i}v_ja_{ij}(s)+v_ib_{2i}(s)\Big )\, ds\le v_i^{-1} \ol{\Ga_i}\, \ol{B_i} e^{D_i(\omega)} v_in_{2i}^{-1}=1.
\end{split}
\end{equation*}
The conclusion is drawn from Theorem \ref{thm3.2}.
\end{proof}

For nonimpulsive systems
\begin{equation}\label{sys_no}
x_i'(t)=-d_i(t) x_i(t) + \sum_{j \neq i} a_{ij}(t) x_j(t) +g_i(t,x_{it}),\q i=1,\dots,n,
\end{equation}
conditions  for the existence of a  positive $\omega$-periodic solution are obtained by taking $\ul{\Ga_i}=\ol{\Ga_i}\equiv \big (e^{D_i(\omega)}-1\big)^{-1}$ and $\ul{B_i}= \ol{B_i}=1$ in the above theorem, leading to:


		\begin{cor}\label {cor3.1} Assume (H4)--(H6). For the matrices in \eqref{D&A},  \eqref{BB} suppose that for some $v>0$:
		\begin{itemize}  
		\item[(a)] either
	$B_2(t)v\le \big [D(t)-A(t)\big ]v\le B_1(t)v;$
		\vskip 0cm
		\item[(b)] or 
$\left\{
\begin{array}{ll}
 \dps \int_0^\omega \Big [B_2(t)+A(t)\Big ]v\, dt\le   \diag\big(1-e^{-D_1(\omega)},\dots,1-e^{-D_n(\omega)}\big)v\\
  \dps\int_0^\omega \Big [B_1(t)+A(t)\Big ]v\, dt\ge \diag \big(e^{D_1(\omega)}-1,\dots,e^{D_n(\omega)}-1\big)v.
  \end{array}
\right.
$
\end{itemize} 
Then, there exists a positive $\omega$-periodic solution of  \eqref{sys_no}
				\end{cor}
\begin{proof} For the system with no impulses  \eqref{sys_no}, we have  $M_1=M_2=I$   and  $N_1^{-1}=\diag\big(e^{D_1(\omega)}-1,\dots,e^{D_n(\omega)}-1\big), N_2^{-1}=\diag\big(1-e^{-D_1(\omega)},\dots,1-e^{-D_n(\omega)}\big)$ for the matrices in Theorem \ref{thm3.3}. 
\end{proof} 

\begin{rmk}\label{rmk3.3} Theorem \ref{thm3.3} and Corollary \ref {cor3.1} take into consideration the sublinear case; for the superlinear case, analised 
in Theorem \ref{thm3.2_super},  similar statements hold.\end{rmk}

From Theorem \ref{thm3.3}, we retrieve some criteria which improve the ones obtained in  \cite{BF} for the particular case of {\it scalar} equations.

\begin{cor}\label{thm3.scalar} Consider the scalar impulsive DDE
\begin{equation}\label{scalar}
\left\{
\begin{array}{ll}
x'(t)=-d(t) x(t) + g(t,x_{t}) \ {\rm for}\ t\ne t_k,  \\
\Delta(x(t_k)):=x(t_k^+)-x_i(t_k)=I_{k} (x(t_k)), \ k \in \Z,
\end{array}
\right. 
\end{equation}
where $(t_k), (I_k)_{k \in \Z}$, 
$d\in C_\omega^+(\R),g:\R\times PC([-\tau,0],\R)\to \R^+$  satisfy the assumptions in (H1)-(H5) (with $i=1$ and $d(t)=d_1(t),g(t,\var)=g_1(t,\var), I_k=I_{1k}$
for $k\in\Z$), and
\begin{itemize}
\item[(H6')]  There are constants $r_0,R_0$ with $0<r_0< R_0 $ and  functions $b_1,b_2\in C_\omega^+(\R)$ with  $\int_0^\omega b_{q}(t)\, dt> 0\, (q=1,2)$, such that for $x\in K$ and $t\in [0,\omega]$,
		\begin{equation}\label{sublinear_scalar}
g(t,x_{t})\geq b_{1}(t)u\q \text{if}\:\:0< u\leq x\leq r_0,\q g(t,x_{t})\leq b_{2}(t)u\q \text{if}\:\:R_0\leq x_i\leq u.		\end{equation}
		\end{itemize}
		With  $\ul{\Ga}=\ul{\Ga_1}, \ol{\Ga}=\ol{\Ga_1}$ and other obvious terminology as above, assume
one of the following conditions:
\begin{itemize}  \item[(a)] $\ol{\Ga}\, \ol{B}(e^{D(\omega)}-1)b_2(t)\le d(t)\le \ul{\Ga}\, \ul{B}(e^{D(\omega)}-1)b_1(t)$;
  \item[(b)]  
  $ \ol{\Ga}\, \ol{B}e^{D(\omega)}\int_0^\omega b_2(t)\, dt\le 1\le \ul{\Ga}\, \ul{B}\int_0^\omega b_1(t)\, dt.$
\end{itemize}
Then, \eqref{scalar} has (at least) one positive $\omega$-periodic solution.
 \end{cor}
 
 \begin{rmk}\label{rmk3.5}For a nonimpulsive scalar equation $x'(t)=-d(t) x(t) + g(t,x_{t})$, the conclusion is obtained by taking
 $\ul{\Ga}=\ol{\Ga}=(e^{D(\omega)}-1)^{-1},\ul{B}= \ol{B}=1$ in Corollary \ref{thm3.scalar}, so that the conditions  read as:
 (a) $b_2(t)\le d(t)\le b_1(t)$ for $t\in [0,\omega]$;
 (b)  $ e^{D(\omega)}\int_0^\omega b_2(t)\, dt\le e^{D(\omega)}-1\le \int_0^\omega b_1(t)\, dt.$
\end{rmk}

\begin{exmp} Consider a  Nicholson's blowflies equation
\begin{equation}\label{N_sca}
x'(t)=-d(t)x(t)+p(t)x(t-\tau(t))e^{-x(t-\tau(t))}
\end{equation}
with $d(t)=\sin^2 t,p(t)=3\cos^2 t$ and $\tau\in C_\pi^+(\R)$.  For $h(x)=xe^{-x}$, one has $h'(0)=1, h(\infty)=0$. Fix any $\vare>0$ small. Clearly, \eqref{sublinear_scalar} is satisfied with $b_1(t)=(1-\vare)p(t)$ and $b_2(t)=\vare$. On the other hand, $\int_0^\pi p(t)\, dt=3\pi/2,e^{\int_0^\pi d(t)\, dt}-1=e^{\pi/2}-1\approx 3.81<3\pi/2$, and consequently $\vare$ can be chosen so that condition (b) in Remark \ref{rmk3.5} holds. Thus, \eqref{N_sca} has a positive $\pi$-periodic solution. Note however that condition $p(t)>d(t)$ is not true for all $t>0$ -- this, according  to an assertion in \cite{AB}, should imply that  there is no positive $\omega$-periodic solution for \eqref{N_sca}, which is contradicted by this example.
%
 \end{exmp}

\begin{rmk}\label{rmk3.4}  In view of both the sublinear and superlinear cases of \eqref{sys}, corresponding to the use of the compressive and expansive forms in Theorem \ref{thmKras}, respectively, the procedure in this section can be applied to impulsive $n$-dimensional DDEs where in \eqref{sys}  the functions $d_i(t),a_{ij}(t),g_i(t,\var)$ are nonpositive (instead of nonnegative), with    (H3)  replaced by $ \prod_{i=1}^p(1+\al_{ik})>\exp\left(\int_0^\omega d_i(t)\, dt\right)$. It is also possible to obtain
 generalisations to systems  of the form 
$$x_i'(t)=-d_i(t) x_i(t)h_i(t,x_i(t)) + \displaystyle{\sum_{j=1,j \neq i}^n a_{ij}(t) x_j(t) +g_i(t,x_{it})},\ i=1,\dots,n$$
(with and without impulses), with $h_i(t,u)$  continuous, bounded above and below by positive constants and $\omega$-periodic in $t$,  by a straightforward adjustement of the present technique. Of course, now the functions  $b_{iq}(t)$ in  (H6) should be multiplied by suitable constants. Details are left to the reader.
  \end{rmk}

\begin{exmp}
Consider the family of differential systems with discrete delays and impulses given by:
\begin{equation}\label{LLJZ0}
 \left\{
\begin{array}{ll}
\dps x_i'(t)=-d_i(t)x_i(t)+\sum_{j\ne i} a_{ij}(t)x_j(t)+\sum_{l=1}^mf_{il}(t,x_i(t-\tau_{il}(t)))\,,&  t\ne t_k,  \\
\Delta(x_i(t_k))=I_{ik} (x_i(t_k)),\q k \in \Z,&  i=1,\ldots,n.
\end{array}
\right. 
\end{equation}
Here, we suppose that  $d_i,a_{ij},\tau_{il}\in C_\omega^+(\R)$, with $d_i(t)>0$,
  $f_{il}(t,u)$ are nonnegative, continuous  and  $\omega$-periodic in $t$, and $t_k,I_{ik}(u)$ satisfy hypotheses (H1)--(H3), $i=1,\dots,n, l=1,\dots, m, k\in \Z$.     

  Define  the values  (in $[0,\infty]$) given by the limits
\begin{equation}\label{limitsJW0}
\begin{split}
\mathfrak{f}_i^0= \liminf_{u\to 0^+}\left(\min_{t\in [0,\omega]}\frac{F_i(t,u)}{d_i(t)u} \right),\q \mathfrak{F}_i^0= \limsup_{u\to 0^+}\left(\max_{t\in [0,\omega]}\frac{F_i(t,u)}{d_i(t)u}\right),\\
\mathfrak{f}_i^\infty= \liminf_{u\to \infty}\left(\min_{t\in [0,\omega]}\frac{F_i(t,u)}{d_i(t)u}\right),\q \mathfrak{F}_i^\infty= \limsup_{u\to \infty}\left (\max_{t\in [0,\omega]}\frac{F_i(t,u)}{d_i(t)u}\right),\end{split}\end{equation}
where $$F_i(t,u)=\sum_l f_{il}(t,u)\q {\rm for}\q i=1,\dots,n,$$ and $\mathfrak{f}^0, \mathfrak{F}^0, \mathfrak{f}^\infty, \mathfrak{F}^\infty$  the diagonal matrices
 with diagonal entries
$\mathfrak{f}_i^0,\mathfrak{F}_i^0, \mathfrak{f}_i^\infty,\mathfrak{F}_i^\infty (1\le i\le n)$, respectively.

%


\begin{thm}\label{thm3.6}  Consider  \eqref{LLJZ0} under  the above assumptions, and 
assume also that:\\
(i)  if $n>1$,   either $ \int_0^\omega a_{ij}(s)\, ds>0$ for all $i\ne j$   or
  $\int_0^\omega F_i(s,0)\, ds>0$, for each  $i=1,\dots,n$;\\
(ii) there exists a vector $v>0$ such that either 
\begin{equation}\label{fF_sub}
M_2[\mathfrak{F}^\infty D(t)+A(t)]v<D(t)v<M_1[\mathfrak{f}^0D(t)+A(t)]v \end{equation}
or
\begin{equation}\label{fF_sup}M_1[\mathfrak{f}^\infty D(t)+A(t)]v>D(t)v>M_2[\mathfrak{F}^0D(t)+A(t)]v, \end{equation}
where $M_1,M_2$ are as in \eqref{mm}.
Then, system \eqref{LLJZ0}
 has at least one positive $\omega$-periodic solution.
\end{thm}
\begin{proof} System  \eqref{LLJZ0} has the form \eqref{sys} with
$g_i(t,x_{it})=\sum_{l=1}^mf_{il}(t,x_i(t-\tau_{il}(t)))$. From (i) and since $f_{il}(t,u)$ are uniformly continuous on bounded sets of $[0,\omega]\times\R$, clearly (H4),(H5) are satisfied.
 

Assume \eqref{fF_sub}, for some $v=(v_1,\dots,v_n)>0$. In particular, this implies that $\mathfrak{f}_i^0>0$.
For any fixed  $\vare\in (0,\mathfrak{f}_i^0)$, let $0<r_0<R_0$ be such that, for  $1\le i\le n$ and $t\in [0,\omega]$, we have
$F_i(t,u)\le (\mathfrak{F}_i^\infty+\vare)d_i(t)u$ for $u\ge R_0$ and $F_i(t,u)\ge (\mathfrak{f}_i^0-\vare)d_i(t)u$ for $0<u\le r_0$. Then, (H6) is satisfied with
$b_{2i}(t)= (\mathfrak{F}_i^\infty+\vare)d_i(t), b_{1i}(t)=(\mathfrak{f}_i^0-\vare)d_i(t).$
Let $\vare$ be sufficiently small so that 
$$m_{2i}[v_i(\mathfrak{F}_i^\infty +\vare)d_i(t)+\sum_j v_ja_{ij}(t)]<v_id_i(t)<m_{1i}[v_i(\mathfrak{f}_i^0-\vare) d_i(t)+\sum_j v_ja_{ij}(t)]$$
  for all $i$ and $t$. The conclusion follows from Theorem \ref{thm3.3}(a).
The superlinear case, where \eqref{fF_sup} holds, is handled in a similar way.
\end{proof}

For  \eqref{LLJZ0} without impulses, as $M_1=M_2=I$ in the above statement, we obtain:

\begin{cor}\label {cor3.3}For the nonimpulsive version of \eqref{LLJZ0},
\begin{equation}\label{LLJZ_no}
 \dps x_i'(t)=-d_i(t)x_i(t)+\sum_{j\ne i} a_{ij}(t)x_j(t)+\sum_{l=1}^mf_{il}(t,x_i(t-\tau_{il}(t)))\,,   i=1,\ldots,n,
\end{equation}
with $d_i,a_{ij},\tau_{il}, f_{il}(t,u)$ as before, assume (i) in the above theorem.
Then, there exists a positive $\omega$-periodic solution if there is $v>0$ such that either 
$$\mathfrak{F}_i^\infty d_i(t)<d_i(t)-\sum_jv_i^{-1}v_ya_{ij}(t)<\mathfrak{f}_i^0d _i(t)\ (1\le i\le n)$$
 or 
 $$\mathfrak{f}_i^\infty d_i(t)>d_i(t)-\sum_jv_i^{-1}v_ya_{ij}(t)>\mathfrak{F}_i^0d_i(t)\ (1\le i\le n).$$ 
 \end{cor}
 Theorem \ref{thm3.6}   recovers the criteria in  \cite{BF}
for the scalar version  of  \eqref{LLJZ0} with a single delay,  
\begin{equation}\label{LLJZ_sc}x'(t)=-d(t)x(t)+f(t,x(t-\tau(t)))\ (t\ne t_k),\q \Delta(x(t_k))=I_{k} (x(t_k))\ (k\in\Z).\end{equation}  
See  also  \cite{Li_et.al,Yan07,ZhangFeng} for the existence of positive periodic solutions for  \eqref{LLJZ_sc}. Note however that Li et al. \cite{Li_et.al} consider the scalar model \eqref{LLJZ_sc} only with {\it nonnegative} impulsive functions $I_k(u)\ge 0$; in this way, the criteria in  \cite{BF,Li_et.al} are not always comparable, as explained in \cite{BF}. 
\end{exmp}

 \section{Systems with  bounded nonlinearities}
%
%
%
%
%
As an illustration with relevant applications, we study some classes of systems with bounded nonlinearities. 
Consider the following two families of  impulsive systems:  \begin{equation}\label{Ex3.1}
  \begin{split}
&x_i'(t)=-d_i(t)x_i(t)+\sum_{j\ne i} a_{ij}(t)x_j(t)\\
&\hskip 10mm +\sum_{l=1}^m \be_{il}(t)  \int_{t-\tau_{il}(t)}^t\ga_{il}(s)h_{il}(s,x_i(s))\,  d_s\nu_{il}(t,s),\  t\ne t_k,\\
&\Delta(x_i(t_k))=I_{ik} (x_i(t_k)),\q k \in \Z,\q   i=1,\ldots,n,
\end{split}
\end{equation}
and
\begin{equation}\label{Ex3.2}
\begin{array}{ll}
x_i'(t)=-d_i(t) x_i(t) + \displaystyle\sum_{j \neq i}a_{ij}(t) x_j(t) + \sum_{l=1}^m \beta_{il}(t)h_{il}\bigg (t, \int_{t-\tau_{il}(t)}^t\!\! \ga_{il}(s)x_i(s) \, d_s\nu_{il}(t,s)\bigg), \\
\Delta(x_i(t_k))=I_{ik} (x_i(t_k)), \ k \in \Z,\q \ i=1,\ldots,n,
\end{array}
\end{equation}
 where:  
\begin{itemize} 
\item[{\bf (h1)}] for  $i,j\in \{1,\dots,n\},l\in \{1,\dots,m\}$,
 $d_i,a_{ij},\tau_{il},\be_{il},\ga_{il}\in C_\omega^+(\R)$, with $\tau_{il}$ bounded, $d_i\not\equiv 0$ and $a_{ij}\not\equiv 0$ for $i\ne j$; 
  $\nu_{il}(t,s)$ are {\it non-decreasing}  in $s$, continuous and  $\omega$-periodic in $t$;
$h_{ik}(t,u)$ are continuous and  $\omega$-periodic in $t$;
\item[{\bf (h2)}]  
for  $i\in \{1,\dots,n\},l\in \{1,\dots,m\}$,
$h_{il}(t,u)$ are {\it bounded} on $\R\times \R^+$ and
 \begin{equation}\label{bi} b_i(t):=\sum_{l=1}^m \be_{il}(t) \int_{t-\tau_{il}(t)}^t\ga_{il}(s)\, d_s\nu_{il}(t,s)> 0;\end{equation}
\item[{\bf (h3)}] for  $i\in \{1,\dots,n\}$, there exist continuous functions $h_i:\R^+\to\R^+$ with   $h_i(0)=0,h_i'(0)=1$, $h_i(u)>0$ for $u>0$, and such that
$$h_{il}(t,u)\ge h_i(u),\q {\rm}\q t\in\R,u\ge 0,l=1,\dots,m;$$ 
\item[{\bf (h4)}] the sequences $(t_k)_{k\in \Z},(I_{ik})_{k\in \Z}$ satisfy (H1)--(H3), $i=1,\dots,n$.
\end{itemize}

Here,   the phase space is $PC=PC([-\tau,0],\R^n)$ with  $\tau=\displaystyle\max_{i,l}\max_{t \in [0,\omega]} \tau_{il}(t)$. However, the situation can be generalised in order to include DDEs with  infinite delay, in which case $t-\tau_{il}(t)$ are replaced by $-\infty$ in the integrals in \eqref{Ex3.1},\eqref{Ex3.2} and  \eqref{bi}.
 
 For $n\times n$ matrix-valued $\omega$-periodic functions $M(t),N(t)$ and $v\in\R^n$, we write
$$M(t)v\le_{\not\equiv}N(t)v$$ if $M(t)v\le N(t)v$ on $ [0,\omega]$ and, for each $i=1,\dots,n$,   there
 is $t_i\in [0,\omega]$ for which
$(M(t_i)v)_i< (N(t_i)v)_i$. The symbol $\ge_{\not\equiv}$ has an analogous meaning.

 \begin{thm}\label{thm3.4} Consider either  \eqref{Ex3.1} or \eqref{Ex3.2}, and assume {\bf (h1)}--{\bf (h4)}. Suppose also that for some $v=(v_1,\dots,v_n)>0$ it holds
 \begin{equation}\label{constantsEx3.1}
 \begin{split}
& \ol{\Ga_i}\, \ol{B_i}\ \max_{t\in[0,\omega]}\int_t^{t+\omega}e^{\int _t^sd_i(r)\, dr}\Big (\sum_{j\ne i}v_i^{-1}v_ja_{ij}(s)\Big )\, ds<1,\\
&\ul{\Ga_i}\, \ul{B_i}\min_{t\in[0,\omega]}\int_t^{t+\omega}e^{\int _t^sd_i(r)\, dr}\Big (\sum_{j\ne i}v_i^{-1}v_ja_{ij}(s)+b_{i}(s)\Big )\, ds>1,\q i=1,\dots,n.
 \end{split}
 \end{equation}
 Then, the system has at least one  positive $\omega$-periodic solution.
   In particular,  for   $M_i,N_i\ (i=1,2)$ as in \eqref{mm}, \eqref{nn} and $B(t)=\diag \, (b_{1}(t),\dots,b_{n}(t))$, this is the case if,  for some $v>0$, either
   \begin{equation}\label{H8Ex3.1}
 M_2 A(t)v\le_{\not\equiv} D(t)v\le_{\not\equiv}M_1 \Big [B(t)+A(t)\Big ]v,
\end{equation}
 or  \begin{equation}\label{H9Ex3.1}
  \int_0^\omega N_2A(t)v\, dt\le v\le \int_0^\omega N_1\Big [B(t)+A(t)\Big ]v\, dt.\\
\end{equation} \end{thm}

\begin{proof} Systems \eqref{Ex3.1}, \eqref{Ex3.2} are particular cases  of   \eqref{sys} with, respectively
\begin{eqnarray}
	g_i(t,x_{it})=\sum_{l=1}^m \be_{il}(t)  \int_{t-\tau_{ik}(t)}^t\ga_{il}(s)h_{il}(s,x_i(s))\,  d_s\nu_{il}(t,s),\label{bdd_distdelay1} \q  i=1,\dots,n, \\
	\nonumber\\
	g_i(t,x_{it})=\sum_{l=1}^m \beta_{il}(t)h_{il}\Big(t, \int_{t-\tau_{il}(t)}^t \ga_{il}(s)x_i(s) \, d_s\nu_{il}(t,s)\Big),\label{bdd_distdelay2}\q  i=1,\dots,n.
	\end{eqnarray}
From the above conditions {\bf (h1)}-{\bf (h2)},   the nonlinearities \eqref{bdd_distdelay1}, \eqref{bdd_distdelay2}  are  {\it bounded} and (H4)-(H5) are fulfilled.  Moreover, the boundedness of all $h_{il}$ also implies that,  in both cases, for any $\vare>0$ there exists $R_0>0$ large, such that
$g_i(t,x_{it})\le \vare u$ for $R_0\le x_i\le u.$

For  \eqref{Ex3.1}, {\bf (h3)}
 implies that,  for any fixed $\vare>0$, there exists $r_0>0$ such that for $0<u\le x_i\le r_0$ we have $h_{il}(s,x_i(s))\ge h_i(u)\ge (1-\vare)u$. Since $\nu_{il}(t,s)$ are nondecreasing in $s$,
we have $g_i(t,x_{it})\ge (1-\vare)b_i(t)u$ for $ u\le x_i\le r_0$. Hence,
 (H6) holds with $b_{1i}(t)=(1-\vare)b_i(t)$ and $b_{21}(t)=\vare$, where $b_i(t)$ is defined in  \eqref{bi}. From \eqref{constantsEx3.1}, for $\vare>0$ sufficiently small we have
\begin{equation}\label{sublinearEx3.1}
 \begin{split}
c_i^0(v)&=\ul{\Ga_i}\, \ul{B_i}\min_{t\in[0,\omega]}\int_t^{t+\omega}e^{\int _t^sd_i(r)\, dr}\Big (\sum_{j\ne i}v_i^{-1}v_ja_{ij}(s)+(1-\vare) b_i(s)\Big )\, ds> 1,\\
C_i^\infty(v)&=\ol{\Ga_i}\, \ol{B_i}\ \max_{t\in[0,\omega]}\int_t^{t+\omega}e^{\int _t^sd_i(r)\, dr}\Big (\sum_{j\ne i}v_i^{-1}v_ja_{ij}(s)+\vare\Big )\, ds< 1,\q i=1,\dots,n,
 \end{split}
 \end{equation} 
hence \eqref{sublinear} holds.
 Thus, Theorem \ref{thm3.2}  provides  the existence of at least one positive periodic solution.
 
 Moreover, in \eqref{BB} we obtain
$
B_1(t)=(1-\vare)B(t), B_2(t)=\vare I.
$   Under conditions \eqref{H8Ex3.1},  note  that
\begin{equation*}
 \begin{split}\min_{t\in[0,\omega]}\int_t^{t+\omega}e^{\int _t^sd_i(r)\, dr}\Big (\sum_{j\ne i}v_i^{-1}v_ja_{ij}(s)+ b_i(s)\Big )\, ds&>m_{1i}^{-1}\min_{t\in[0,\omega]}\int_t^{t+\omega}e^{\int _t^sd_i(r)\, dr}d_i(s)\, ds=(\ul{\Ga_i}\, \ul{B_i})^{-1},\\
\max_{t\in[0,\omega]}\int_t^{t+\omega}e^{\int _t^sd_i(r)\, dr}\Big (\sum_{j\ne i}v_i^{-1}v_ja_{ij}(s)\Big )\, ds&<m_{2i}^{-1}\max_{t\in[0,\omega]}\int_t^{t+\omega}e^{\int _t^sd_i(r)\, dr}d_i(s)\, ds=(\ol{\Ga_i}\, \ol{B_i})^{-1},\end{split}
 \end{equation*}
thus  one can find $\vare>0$ small enough so that conditions  \eqref{sublinearEx3.1} hold.
 In an analogous way,  since
 \begin{equation*}
 \begin{split}
 & \int_t^{t+\omega}e^{\int _t^sd_i(r)\, dr}\Big (\sum_{j\ne i}v_i^{-1}v_ja_{ij}(s)\Big )\, ds<e^{D_i(\omega)}\int_t^{t+\omega}\Big (\sum_{j\ne i}v_i^{-1}v_ja_{ij}(s)\Big )\, ds,\\
&\int_t^{t+\omega}e^{\int _t^sd_i(r)\, dr}\Big (\sum_{j\ne i}v_i^{-1}v_ja_{ij}(s)+b_{i}(s)\Big )\, ds>\int_t^{t+\omega}\Big (\sum_{j\ne i}v_i^{-1}v_ja_{ij}(s)+b_{i}(s)\Big )\, ds
\end{split}
 \end{equation*} for all  $i$ and $t\in[0,\omega]$, conditions  \eqref{sublinearEx3.1} follow under \eqref{H9Ex3.1}. 
%
%
%
%

 Similarly, for  \eqref{Ex3.2}, again using {\bf(h3)} and the fact that $\nu_{il}(t,s)$ are nondecreasing in $s$, for $r_0>0$ sufficiently small  we derive 
 \begin{equation*}
\begin{split}
g_i(t,x_{it})&\ge\sum_{l=1}^m \beta_{il}(t)h_i\bigg ( \int_{t-\tau_{il}(t)}^t \ga_{il}(s)x_i(s) \, d_s\nu_{il}(t,s)\bigg)\\
&\ge (1-\vare)u\sum_{l=1}^m \beta_{il}(t) \int_{t-\tau_{il}(t)}^t \ga_{il}(s) \, d_s\nu_{il}(t,s)\\
&=(1-\vare)b_i(t)u,\q {\rm for}\q u\le x_i\le r_0.
\end{split}
\end{equation*}
This means that (H6) still holds with the same $b_{1i}(t)=(1-\vare)b_i(t)$ and $b_{21}(t)=\vare$. The statements follow again by the previous results.\end{proof}

From   Corollary \ref {cor3.1}, we also conclude:

\begin{cor}\label {cor3.4} Assume {\bf (h1)}--{\bf (h3)} and define $B(t)=\diag \, (b_{1}(t),\dots,b_{n}(t))$. If there is a vector $v=(v_1,\dots,v_n)>0$ such that:\\
(a) either $A(t)v< D(t)v$ and 
\begin{equation}\label{ga_i}
\ga_i(t,v):=\frac{b_i(t)v_i}{d_i(t)v_i-\sum_{j} v_ja_{ij}(t)}\ge_{\not\equiv}1,\q t\in [0,\omega], i=1,\dots,n,
\end{equation}
(b) or \begin{equation}\label{H9for0}
e^{\int_0^\omega d_i(s)\, ds} \int_0^\omega \sum_{j} v_ja_{ij}(t)dt\le v_i (e^{\int_0^\omega d_i(s)\, ds}-1) \le  \int_0^\omega \big [b_i(t)v_i+\sum_{j} v_ja_{ij}(t)\big]\, dt,\ i=1,\dots,n,
\end{equation} 
then the nonimpulsive system \eqref{sys_no} where $g_i(t,x_{it})$ has one of the forms  \eqref{bdd_distdelay1}, \eqref{bdd_distdelay2}
possesses at least one positive $\omega$-periodic solution.
\end{cor}


 Note that (a) is equivalent to $A(t)v< D(t)v\le_{\not\equiv} \big [B(t)+A(t)\big ]v$. For the case of {\it bounded} delay, we emphasize that the existence of such a periodic solution for the nonimpulsive version of system  \eqref{Ex3.1} was established  in \cite[Theorem 3.1]{Faria17} assuming {\bf (h1)}--{\bf (h3)} and  that there are constants $\al,\ga $ such that
 \begin{equation}\label{H5*}
1<\al \le \ga_i(t,v)\le \ga,\q t\in [0,\omega], i=1,\dots,n.
\end{equation}
Under the above general assumptions  and the stronger requirement \eqref{H5*},  the nonimpulsive version of system  \eqref{Ex3.1} was proven to be permanent \cite{Faria21}. Moreover, in this setting, a fixed point argument allows to conclude  that the sufficient conditions for permanence also  imply the existence of a positive periodic solution  \cite{Faria17,Zhao08}. However, at least for non-periodic systems, condition \eqref{ga_i} is not enough to guarantee the persistence of the system. In fact, in \cite[Example 4]{Faria21}  a counter-example was given: a nonimpulsive  system of the form \eqref{Ex3.1} with $n=2$ and a single discrete delay, which is not persistent although it satisfies conditions {\bf (h1)-(h3)} and $ \ga_i(t,{\bf 1})>1$ for all $t\in \R^+$ and $i=1,2$.



 \section{Applications to natural sciences models}
\setcounter{equation}{0}

In this section,  our results are illustrated with applications to some selected  hematopoiesis  models and Nicholson-type systems. Several other models appearing in  natural sciences could have been analysed. 

  \begin{exmp}\label{exmp4.1} {\it A hematopoiesis-type model.}

In the celebrated paper of Mackey and Glass \cite{MG},  two basic scalar  models of the form $x'(t)=-dx(t)+f(x(t-\tau))$ were proposed to describe the hematopoiesis process (production and specialisation of blood cells) taking place in  the bone marrow, with   either a monotone decreasing  production function $f(x)=b\frac{a^\al}{a^\al+x^\al}$, or   a unimodal  function $f(x)=b\frac{x}{a^\al+x^\al}$, for $d,a,b,\al,\tau>0$ (after normalization of the coefficients, one can take $a=1$). Numerous generalizations of the original Mackey-Glass equations have been analysed in the literature, including impulsive periodic models with several delays. These models are relevant in applications, as they  integrate a periodic environmental variation, as well as
abrupt changes in the regulation process -- incorporated as impulses --, due to e.g.~radiation or drug administration.
 See e.g. \cite{KongNietoFu,LiuYanZhang}, also for further references, and \cite{Roussel} for an interesting historical insight on Mackey-Glass models.  
 
 We next apply our main results to  a multidimensional version of a hematopoiesis model with  impulses and distributed  delays:
\begin{equation}\label{hemato}
\left\{
\begin{array}{ll}
x_i'(t)=\displaystyle{-d_i(t) x_i(t) + \sum_{j \neq i}a_{ij}(t) x_j(t) + \sum_{l=1}^m \dfrac{\beta_{il}(t)}{1+c_{il}(t) \left(\int_{t-\tau_{il}(t)}^t x_i(s) \, ds\right)^{\alpha_{il}}}},\q t\ne t_k\\
x_{i}(t_k^+)-x_i(t_k)=I_{ik} (x(t_k)), \ k \in \Z,\q  \ i=1,\ldots,n.
\end{array}
\right.
\end{equation}

\begin{thm}\label{thm4.1} Assume that the sequence $(t_k)$ and the impulsive functions $I_{ik}:\R^+\to\R$ satisfy (H1)--(H3), that $d_i,a_{ij},  \beta_{il},\tau_{il},c_{il}\in C_\omega^+(\R)$ with $d_i\not\equiv 0, \sum_{l=1}^m\be_{il}(t)\not\equiv 0,\tau_{il}(t)\in [0,\tau]$, $\alpha_{il}$ are positive constants, for all $i=1,\ldots,n$ and $l=1,\ldots,m$.  If there is $v=(v_1,\dots,v_n)>0$ such that
  \begin{equation}\label{Linfty=0} \ol{\Ga_i}\, \ol{B_i}\max_{t\in[0,\omega]}\int_t^{t+\omega} e^{\int _t^sd_i(r)\, dr}\Big (\sum_{j\ne i}v_i^{-1}v_ja_{ij}(s)\Big )\, ds<1,
  \end{equation}
 then there is a positive $\omega$-periodic solution of \eqref{hemato}. In particular, this is the case if, for some vector $v>0$, either
$$\diag\big (\ol{\Ga_1}\, \ol{B_1}(e^{D_1(\omega)}-1),\dots,\ol{\Ga_n}\, \ol{B_n}(e^{D_n(\omega)}-1)\big)A(t)v\le_{\not\equiv} D(t)v,$$
or 
$$\diag\big (\ol{\Ga_1}\, \ol{B_1}e^{D_1(\omega)},\dots,\ol{\Ga_n}\, \ol{B_n}e^{D_n(\omega)}\big)
 \int_0^\omega A(t)v\, dt\le  v.$$
\end{thm}

\begin{proof}
This equation falls into the framework of problem \eqref{sys}, with $g_i(t,\var_i)= \sum_{l=1}^m h_{il}(t,\var_i)$ and
$$
 h_{il}(t,\var_i)=\dfrac{\beta_{il}(t)}{1+c_{il}(t)  \left(\int_{-\tau_{il}(t)}^0 \var_i(s) \, ds\right)^{\alpha_{il}}}, \ i=1,\ldots,n,\ l=1,\ldots,m,
$$
so $g_i$ are bounded functions.
Since $\int_0^\omega g_i(t,0)\, dt= \sum_{l=1}^m \int_0^\omega \beta_{il}(t)\, dt>0$, clearly (H4)(ii) is satisfied. On the one hand,
for $i,l$ fixed and all $t\in [0,\omega], x,y\in K$,
\begin{equation*}
|h_{il}(t,x_{it})-h_{il}(t,y_{it})| =\beta_{il}(t) \left |\frac{1}{1+c_{il}(t)X^{\al_{il}}}-\frac{1}{1+c_{il}(t)Y^{\al_{il}}}\right |,
\end{equation*}
where $X=\int_{-\tau_{il}(t)}^0 x_i(t+s) \, ds, Y=\int_{-\tau_{il}(t)}^0y_i(t+s) \, ds$,
and 
$|X-Y|\le \int_{-\tau_{il}(t)}^0 |x_i(t+s)- y_i(t+s)|\, ds\le \tau \|x-y\|.$
Futhermore,  for each function of the form $f(t,x):=\frac{1}{1+c(t)x^\al}, t\in [0,\omega], x\in\R^+,$ with  $c(t)\in C_\omega^+(\R)$ and $\al>0$, from its uniform continuity on compact sets  we derive that
for any $M>0$  and $\vare>0$, there is $\de>0$ such that, for $t\in [0,\omega], x,y\in [0,M]$ with $|x-y|<\de$, we have
$|f(t,x)-f(t,y)|\le \vare$. In this way, we conclude that
   the nonlinearities satisfy hypothesis (H5).

If \eqref{Linfty=0} is satisfied for some $v>0$, it is possible to choose $0<\vare <M$ such that,
%
 with $b_{i1}(t)\equiv M,b_{i2}(t)\equiv \vare$ in the definition of the constants $C_i^\infty(v),c_i^0(v)$  as in
\eqref{sublinear}, we have
 \begin{equation}\label{5.2}C_i^\infty(v)<1<c_i^0(v),\q i=1,\dots,n.\end{equation}
 Since $0< g_i(t,x_{it})\le g_i(t,0)$,    for any $0<\vare<M$  there are $0<r_0<R_0$ such that $g_i(t,x_{it})\ge M u$ if $0\le u\le x_i\le r_0$ and $g_i(t,x_{it})\le \vare u$ if $R_0\leq x_i\leq u$, for $t\in\R,x\in X^+, i=1,\dots,n$. Thus, (H6) is satisfied with the above  choices $b_{i1}(t)\equiv M$ and $b_{i2}(t)\equiv \vare$. The results are a consequence of \eqref{5.2} and Theorems \ref{thm3.2} and \ref{thm3.3}.
\end{proof}

Clearly, an analogous statement applies to a periodic hemotopoiesis system where the nonlinearities $g_i(t,\var_i)$ contain only discrete delays, so that 
$g_i(t,x_{it})= \sum_{l=1}^m \frac{\beta_{il}(t)}{1+c_{il}(t)  x_i^{\alpha_{il}}(t-\tau_{il}(t))}$, for coefficients and delay functions as in the above theorem.

%

  In \cite{FO19}, the following scalar hematopoiesis model with linear impulses and discrete delays was considered:
\begin{equation}\label{hemato2}
\begin{cases}
x'(t)=\displaystyle{-d(t) x(t)+ \sum_{l=1}^m \dfrac{\beta_{l}(t)}{1+c_{l}(t) x^{\alpha_l}(t-\tau_l(t))}},\ \ t\ne t_k,\\
x(t_k^+)-x(t_k)=b_{k} x(t_k), \ k \in \Z,
\end{cases}
\end{equation}
 where $(b_k),(t_k)$ are $\omega$-periodic sequences, with  $0\le t_1<\cdots <t_p< \omega$ for some $p$, $\alpha_l$ are positive constants, $d,\be_l,\tau_l,c_l\in C_\omega^+(\R)$
  and $d \not\equiv 0, \sum_l\beta_l \not\equiv 0$, $c_l(t)>0$, for all $t\in [0,\omega],l=1,\ldots,m$.
From the version of Theorem \ref{thm4.1} for discrete delays, we conclude the existence of a positive periodic solution for \eqref{hemato2} if $b_k>-1$ and $ \prod_{k=1}^p(1+b_k)<e^{\int_0^\omega d(t)\, dt}$, recovering the result in \cite[Theorem 3.1]{FO19}.

\end{exmp}

 \begin{exmp}\label{exmp4.2} {\it Nicholson blowflies  systems.}

Recently, there has been an increasing interest in  periodic (or almost periodic) Nicholson-type systems with patch structure, and several authors have addressed the topics of existence, uniqueness and/or exponential stability of (almost) positive periodic solutions see e.g.  \cite{DingFu,Faria17,Faria20, HWH,Troib,Wang+19,Xu}.

Here, we consider a  generalised 
Nicholson system with distributed delays given by
\begin{equation}\label{N}
\begin{cases}
\dps x_i'(t)=-d_i(t)x_i(t)+\sum_{j\ne i} a_{ij}(t) x_j(t) +\dps \sum_{l=1}^{m} \be_{il}(t)  \int_{t-\tau_{il}(t)}^t\!\!  \ga_{il}(s)x_i(s) e^{-c_{il}(s)x_i(s)}\, ds,\q t\ne t_k\\
x_{i}(t_k^+)-x_i(t_k)=I_{ik} (x_i(t_k)), \ k \in \Z,\q i=1,\dots,n.
\end{cases}
\end{equation}

\begin{thm}\label{thmN1} Assume that $(t_k),I_{ik}(u)$ satisfy (H1)--(H3), 
$d_i,a_{ij}, \beta_{il} , c_{il}, \ga_{il},\tau_{il}\in C_\omega^+(\R)$  with $d_i\not\equiv 0, a_{ij}\not\equiv 0\, (j\ne i),\sum_l\beta_{il}\not\equiv 0$,  $c_{il}(t)>0$, $0\le \tau_{il}(t)\le \tau$ on $[0,\omega]$, for some $\tau>0$. Assume also that either \eqref{constantsEx3.1}, \eqref{H8Ex3.1} or \eqref{H9Ex3.1} is satisfied with
$$b_i(t)=\sum_{l=1}^{m} \be_{il}(t)\int_{t-\tau_{il}(t)}^t\!\!  \ga_{il}(s)\, ds,\q t\ge 0,\q i=1,\dots,n.$$
Then  \eqref{N} has a positive $\omega$-periodic solution.
\end{thm}

\begin{proof}
Note that \eqref{N} has the form \eqref{Ex3.1} with $h_{il}(s,u)=ue^{-c_{il}(s)u}, \nu_{il}(t,s)=s$. Let $c_{il}^+=\max_{t\in[0,\omega]}c_{il}^+(t)$ and
$c_i^+=\max_l c_{il}^+, i=1,\dots,n, l=1,\dots, m$. Then (h3) is satisfied with $h_i(u)=ue^{-c_i^+u}$. The result is a consequence of the criteria in Theorem \ref{thm3.4}.
\end{proof}

A similar result holds for e.g.~Mackey-Glass-type systems with patch structure  \eqref{sys}   with 
$$g_i(t,x_{it})= \sum_{l=1}^m\frac{\beta_{il}(t)\, \int_{t-\tau_{il}(t)}^t x_i(s) \, ds}{1+c_{il}(t) \left(\int_{t-\tau_{il}(t)}^t   x_i(s) \, ds \right)^{\alpha_{il}}},$$
where $\alpha_{il}$ are positive constants,
$d_i,a_{ij}, \beta_{il} , c_{il}, \tau_{il}\in C_\omega^+(\R)$, which are included in the family \eqref{Ex3.2}. Theorem \ref{thm3.4} is applicable with $b_i(t)$ in \eqref{bi} given by
$b_i(t)=\sum_{l=1}^{m} \be_{il}(t)\tau_{il}(t),\, i=1,\dots,n.$
\end{exmp}

 \begin{exmp}\label{exmp4.5} {\it Nicholson  systems with mixed monotonicity.}
 
 For the last years, there has been an increasing interest in DDEs with the nonlinearities given by functions $f(t,x,y)$ with mixed monotonicity in the spatial variables, i.e., with $f$ increasing in the  variable $x$ and  decreasing in $y$. See e.g.~\cite{bb2016} for some relevant features and applications of  scalar  DDEs with mixed monotonicity.
 
 In the innovative work of Chen \cite{Chen}, a criterion for the existence of a positive periodic solution for the periodic Nicholson equation 
 \begin{equation}\label{eqChen0}    x'(t)=- d(t) x(t)+ b(t) x(t-\tau(t))e^{-c(t) x(t-\th(t))}  \end{equation}
was established. Such criterion was generalised and improved in \cite{FO19} for impulsive Nicholson equations with multiple pairs of discrete delays,
\begin{equation}\label{Nscalar_mix}
 \dps x'(t)=-d(t)x(t)+\dps \sum_{l=1}^{m} \be_{l}(t) x(t-\tau_{l}(t))e^{-c_{l}(t)x(t-\th_{l}(t))},\end{equation}
 and linear impulses, and in \cite{BF} for the case of distributed delays and more general impulses as in \eqref{sys}. The aim here is to state a  result for systems, as an illustration of Theorem \ref{thm3.2} with the coefficients  $\sigma_i$ in the definition of the cone $K$ having an  active role.
 
  We start with no impulses, and consider  periodic Nicholson's blowflies systems with mixed monotonicity  as follows:
\begin{equation}\label{N_mix}
 \dps x_i'(t)=-d_i(t)x_i(t)+\sum_{j\ne i} a_{ij}(t) x_j(t)
+\dps \sum_{l=1}^{m} \be_{il}(t) x_i(t-\tau_{il}(t))e^{-c_{il}(t)x_i(t-\th_{il}(t))},\ i=1,\dots,n,
\end{equation}
where $d_i,a_{ij}, \beta_{il} , c_{il},\tau_{il},\th_{il}\in C_\omega^+(\R)$  with $d_i\not\equiv 0, a_{ij}\not\equiv  0,\sum_l\beta_{il}\not\equiv 0$,  $c_{il}(t)>0$ on $[0,\omega]$. Define
$$b_i(t)=\sum_{l=1}^{m} \be_{il}(t),\q i=1,\dots,n.$$

Consider the cone $K=K(\sigma)$, where $\sigma=(\sigma_1,\dots,\sigma_n)$ and $\sigma_i=e^{-\int_0^\omega d_i(s)\, ds}=e^{-D_i(\omega)}$.  Below, we  consider the functions $g_i(t,\var_{i})= \sum_{l=1}^{m} \be_{il}(t) \var_i(-\tau_{il}(t))e^{-c_{il}(t)\var_i(-\th_{il}(t))}$ and  the positive constants 
$c_{il}^+=\max_{t\in[0,\omega]}c_{il}^+(t), \, c_{il}^-=\min_{t\in[0,\omega]}c_{il}^+(t)$ and
$c_i^+=\max_l c_{il}^+, 
\, c_i^-=\min_l c_{il}^-\ (1\le l\le m,1\le i\le n)$. 

Now, let $x=(x_1,\dots,x_n)\in K$, and recall that $x_i\ge \sigma_i \|x_i\|$ for all $i$.
 Then, for any $\vare>0$, there are $0<r_0<R_0$ such that
\begin{equation}
\begin{split}
&g_i(t,x_{it})\le \sum_{l=1}^{m} \be_{il}(t)\|x_i\|e^{-c_i^-\sigma_i\|x_i\|}< \vare u \q {\rm if}\q R_0\le x_i\le u,\\
&g_i(t,x_{it})\ge \sum_{l=1}^{m} \be_{il}(t)\sigma_i\|x_i\|e^{-c_i^+\|x_i\| }> 
(1-\vare)\sigma_i b_i(t)u \q {\rm if}\q 0<u\le x_i\le r_0.
\end{split}
\end{equation}
Hence,  (H6) is satisfied with $b_{1i}(t)=(1-\vare)\sigma_i b_i(t)$ and $b_{2i}(t)=\vare.$ Reasoning as in the proof of Theorem \ref{thm3.4}, we obtain the following result:

\begin{thm}\label{thm4.4} Consider \eqref{N_mix}, with  all the coefficients and delays satisfying the above general conditions. With $D_i(\omega)=\int_0^\omega d_i(t)\, dt, b_i(t)=\sum_{l=1}^{m} \be_{il}(t)\, (1\le i\le n)$,
assume that there is a vector $v=(v_1,\dots,v_n)>0$ such that one of the following conditions is satisfied:
\item[(a)] 
$\left\{
\begin{array}{ll}
 \int_t^{t+\omega}e^{\int _t^sd_i(r)\, dr}\big (\sum_{j\ne i}v_i^{-1}v_ja_{ij}(s)\big) ds<e^{D_i(\omega)}-1,\\
 \int_t^{t+\omega}e^{\int _t^sd_i(r)\, dr}\Big (\sum_{j\ne i}v_i^{-1}v_ja_{ij}(s)+e^{-D_i(\omega)}b_i(s)\Big ) ds>e^{D_i(\omega)}-1, \ t\in[0,\omega],i=1,\dots,n;
\end{array}
\right.
$
\item[(b)] $0\le_{\not\equiv} v_id_i(t)-\sum_{j\ne i} v_ja_{ij}(t)\le_{\not\equiv}
v_ie^{-D_i(\omega)}b_i(t),\ t\in[0,\omega],i=1,\dots,n$;
	
\item[(c)] 
$\left\{
\begin{array}{ll}
 \int_0^\omega \sum_{j\ne i} v_ja_{ij}(t)\, dt\le  v_i\big(1-e^{-D_i(\omega)})\\
 v_ie^{-D_i(\omega)} \int_0^\omega b_i(t)\, dt+\sum_{j\ne i} v_j\int_0^\omega a_{ij}(t)\, dt\ge v_i\big(e^{D_i(\omega)}-1\big),\q i=1,\dots,n.
  \end{array}
\right.
$\\
Then \eqref{N_mix} has at least one positive $\omega$-periodic solution.
\end{thm}


\begin{cor}\label{cor4.2} Consider the periodic Nicholson equation with multiple pairs of
delays \eqref{Nscalar_mix}, 
where $d,\beta_{l} , c_{l},\tau_{l},\th_{l}\in C_\omega^+(\R)$  with $d\not\equiv 0, \sum_l\beta_{l}\not\equiv 0$,  $c_{l}(t)>0$ on $[0,\omega]$. Assume that one of the following conditions is satisfied:
\item[(a)] 
$
  \dps \sum_{l=1}^{m} \int_t^{t+\omega} \be_{l}(s)e^{-\int_s^{t+\omega} d(r)\, dr}\, ds\ge e^{\int_0^\omega d(s)\, ds}-1, \ t\in [0,\omega]$;\\
\item[(b)] $\dps 
\sum_{l=1}^{m} \be_{l}(t)\ge_{\not\equiv} d(t)e^{\int_0^\omega d(s)\, ds},\ t\in [0,\omega]$;
	
\item[(c)] 
$
  \dps \sum_{l=1}^{m} \int_0^\omega \be_{l}(t)\, dt\ge e^{\int_0^\omega d(s)\, ds}\big(e^{\int_0^\omega d(s)\, ds}-1\big).
$\\
Then,  \eqref{Nscalar_mix} has at least one positive $\omega$-periodic solution.
\end{cor}


 Even for the scalar case, Corollary \ref{cor4.2} improves results in \cite{Chen, FO19}. In fact, Chen  \cite{Chen} showed that
  a positive $\omega$-periodic solution for \eqref{eqChen0}  exists if $$
  \int_0^\omega b(t)\, dt> e^{2\int_0^\omega d(s)\, ds}{\int_0^\omega d(s)\, ds}
$$ and in
\cite{FO19}  the existence of a positive $\omega$-periodic solution for \eqref{Nscalar_mix} was established under the strict inequalities ``$>$" in (a),(b) or (c) above.
 Clearly,   for impulsive systems \begin{equation}\label{NImp_mix} \begin{cases}
\dps x_i'(t)=-d_i(t)x_i(t)+\sum_{j\ne i} a_{ij}(t) x_j(t)
+\dps \sum_{l=1}^{m} \be_{il}(t) x_i(t-\tau_{il}(t))e^{-c_{il}(t)x_i(t-\th_{il}(t))},\ t\ne t_k\\
x_{i}(t_k^+)-x_i(t_k)=I_{ik} (x_i(t_k)), \ k \in \Z,\q i=1,\dots,n,
\end{cases}
\end{equation}
 similar criteria can be stated as in Theorem \ref{thm3.4}, with each $\sigma_i$ above replaced by $\sigma_i= \ul{B_i}\ol{B_i}^{-1}e^{-D_i(\omega)}$.

\end{exmp}

\begin{exmp}\label{exmp4.3} {\it  A planar Nicholson system with discrete delays.}

Consider
 \begin{equation}\label{Nplanar_Imp}
\begin{cases}
&\dps x_1'(t)=-d_1(t)x_1(t)+ a_1(t) x_2(t) +\dps \sum_{l=1}^{m} \be_{1l}(t) x_1(t-\tau_{1l}(t))e^{-c_{1l}(t)x_1(t-\tau_{1l}(t))}
,\ t\ne t_k,\\
&\dps x_2'(t)=-d_2(t)x_2(t)+ a_2(t) x_1(t) +\dps \sum_{l=1}^{m} \be_{2l}(t) x_2(t-\tau_{2l}(t))e^{-c_{2l}(t)x_2(t-\tau_{2l}(t))}
,\ t\ne t_k,\\
&x_{i}(t_k^+)-x_i(t_k)=I_{ik} (x_i(t_k)), \ k \in \Z,\q i=1,2
\end{cases}
\end{equation}
where all the coefficients and delays are in $C_\omega^+(\R)$, $c_{il}(t)>0$, $ d_i(t)>0, \int_0^\omega a_i(t)\, dt>0$ and, as before, define $b_i(t)=\sum_{l=1}^{m} \be_{il}(t),\, i=1,2$. 
Theorem \ref{thm3.6} leads to the criterion below.
 
 \begin{thm}\label{thm4.2} Under the above conditions, suppose that $t_k,I_{1k},I_{2k}$ satisfy (H1)-(H3) and that,  for $m_{qi},\, q,i=1,2$ defined as in \eqref{mm}, there is $v=(v_1,v_2)>0$ such that:\\
    \begin{equation} \label{4.7}
    \begin{split}
  &m_{21}max_{t\in[0,\omega]} \frac{v_1^{-1}v_2a_1(t)}{d_1(t)}<1<m_{11}\min_{t\in[0,\omega]} \frac{b_1(t)+v_1^{-1}v_2a_1(t)}{d_1(t)}\, ,\\
&m_{22}max_{t\in[0,\omega]} \frac{v_2^{-1}v_1a_2(t)}{d_2(t)}<1<m_{12}\min_{t\in[0,\omega]} \frac{b_2(t)+v_2^{-1}v_1a_2(t)}{d_2(t)}\, .
\end{split}
  \end{equation}
Then there exists a positive $\omega$-periodic solution of \eqref{Nplanar_Imp}.
 \end{thm}
  
  \begin{proof} With the notation in \eqref{limitsJW0}, for $v=(v_1,v_2)>0$
  we have $\mathfrak{f}_i^0=\min_{t\in[0,\omega]}\frac{b_i(t)}{d_i(t)},\mathfrak{F}_i^\infty=0\, (i=1,2)$.  Conditions \eqref{4.7} imply that the requirements in \eqref{fF_sub} are satisfied.
   \end{proof}  
   
  For the planar system with no impulses, with the particular choice  of $v=(1,1)$, conditions \eqref{4.7} reduce to
$\min_{t\in[0,\omega]} \frac{a_i(t)}{d_i(t)}<1< \min_{t\in[0,\omega]} \frac{b_i(t)+a_i(t)}{d_i(t)}$ for $ i=1,2.$  \end{exmp}
%

\begin{rmk}   We observe that a very particular case of \eqref{Nplanar_Imp} was considered by Zhang et al.  \cite{ZHW}, under the following requirements:
   
  (i) the impulses are {\it linear}, $I_{ik}(u)=\eta_{ik}u\, (i=1,2,k\in\Z)$ with $\eta_{ik}>-1$, and 
	(H1) holds; 

 (ii) the functions $t\mapsto \prod_{k:t_k\in[0,t)}(1+\eta_{ik}) \, (i=1,2)$  are $\omega$-periodic;
 
 (iii) $d_i(t), a_i(t), \be_{il}(t),c_{il}(t),\tau_{il}(t)$ are strictly positive functions in $C_\omega(\R)$, for $i=1,2, l=1,\dots,m$;
 
 (iv)
   \begin{equation}\label{hypZHW}\frac{a_1^+a_2^+}{d_1^-d_2^-}<1\end{equation}
    where $a_i^+=\max_t a_i(t), d_i^-=\min_t d_i(t)$.

Recall that for linear impulses as in (i) above,  we have $J_{ik}(u)\equiv (1+\eta_{ik})^{-1}$, $B_i(t):=B_i(t;u)=\prod_{k:t_k\in[0,t)}(1+\eta_{ik})^{-1}$ and $\Ga_i(u)\equiv (B_i(\omega)e^{D_i(\omega)}-1)^{-1}$, in particular these functions do not depend on $u\in\R^+, i=1,2$. Moreover,   from (ii) it follows  from Liu and Takeuchi  \cite{LiuTakeuchi} that $\prod_{k=1}^p(1+\eta_{ik})^{-1}=1$, thus $B_i(\omega)=1$ and  $m_{1i}=\ul{B_i},m_{21}=\ol{B_i}$.

  We also stress that in \cite{ZHW} the authors reduce the system to a system without impulses and  nonlinearities with jumps, by the change of variables \eqref{contSol}.
  However, in \cite{ZHW}  initial conditions $x_t=\phi$ are taken with $\phi=(\phi_1,\phi_2)$ strictly positive and continuous, instead of piecewise continuous functions with jumps discontinuities at the instants $t_k$ -- which seems not to be consistent with the problem. In this scenario,  by using a Krasnoselskii's fixed point argument, Zhang et al.  \cite{ZHW}
  claimed the existence of a positive $\omega$-periodic solution, without imposing any other restrictions on the impulses.
        
    Note that \eqref{hypZHW} implies that it is possible to choose $v=(v_1,v_2)>0$ such that
    $$v_2\frac{a_1^+}{d_1^-}<v_1<v_2\frac{d_2^-}{a_2^+},$$ and therefore
    $$\max_{t\in[0,\omega]} \frac{v_1^{-1}v_2a_1(t)}{d_1(t)}\le v_1^{-1}v_2\frac{a_1^+}{d_1^-}<1,\ 
\max_{t\in[0,\omega]} \frac{v_2^{-1}v_1a_2(t)}{d_2(t)}\le v_2^{-1}v_1\frac{a_2^+}{d_2^-}<1.$$
In particular, for the nonimpulsive situation,  the first inequalities in both conditions \eqref{4.7} are satisfied.
  On the other hand, contrary to what is asserted in \cite{ZHW},  the above impositions (i)-(iv) are not enough to guarantee the existence of a positive periodic solution, as the following simple counter-example shows. More elaborated examples for nonautonomous systems and with nonlinear impulses  could also be given.
  \end{rmk}
  
\begin{exmp}\label{exmp4.4} {\it  An autonomous  planar Nicholson system with and without impulses.}  

Consider the autonomous nonimpulsive planar system
  \begin{equation}\label{Nplanar_Counter}
\begin{cases}
x_1'(t)=-d_1x_1(t)+ a_1 x_2(t) + \be_{1} x_1(t-\tau_1)e^{-c_1x_1(t-\tau_1)}\\
x_2'(t)=-d_2x_2(t)+ a_2 x_1(t) +\be_2 x_2(t-\tau_2)e^{-c_2x_2(t-\tau_2)}
\end{cases}
\end{equation}
with $d_i,a_i,\be_i,c_i,\tau_i>0,\, i=1,2$. For this system, condition  \eqref{hypZHW}  in \cite{ZHW} reduces to $d_1d_2>a_1a_2$. Define the so-called community matrix as
$M=\left [\begin{matrix}
\be_1-d_1&a_1\\a_2&\be_2-d_2
\end{matrix}\right]$.
From \cite{FariaRost} it follows that $s(M)\le 0$ is a necessary and sufficient condition for 0 to be a globally asymptotically stable equilibrium of \eqref{Nplanar_Counter} (in the set of all nonnegative solutions), where $s(M)=\max\{ Re\, \la:\la\in\sigma (M)\}$. Choose e.g. $d_i=2, a_i=\be_i=1,\, i=1,2$. Clearly  \eqref{hypZHW}  is satisfied. However, since  $\sigma (M)=\{0,-2\}$,  0 is a global attractor for \eqref{Nplanar_Counter} -- in particular, the claim in \cite{ZHW}  is not valid for the nonimpulsive case.

On the other hand,  fix  any $\omega >0$ and add  to system \eqref{Nplanar_Counter} e.g. a single    linear, constant, positive  impulse on each component and on each interval of length
 $\omega$:
\begin{equation}\label{Nplanar_Counter_Imp}\Delta x_i(t_k)=\eta_i x_i(t_k),\q i=1,2, k\in \Z
\end{equation}
with $0< t_1< \omega, t_{k+1}=t_k+\omega, k\in\Z$, and $0<\eta_i<e^{2\omega}-1.$  In this situation, (H1)-(H3) hold.  With the previous notation we have
$\ul{\Ga_i}=\ol{\Ga_i}=\Big((1+\eta_i)^{-1}e^{2\omega}-1\Big)^{-1}, \,
  \ul{B_i}=(1+\eta_i)^{-1}, \ol{B_i}=1$,
   \begin{equation}\label{mm_eta}m_{1i}=m_{1i}(\eta_i)=\frac{e^{2\omega}-1}{e^{2\omega}-(1+\eta_i)},\ m_{2i}=m_{2i}(\eta_i)=\frac{e^{2\omega}-1}{(1+\eta_i)^{-1}e^{2\omega}-1},\ i=1,2.\end{equation}
   Note that \eqref{4.7} is satisfied with $v=(1,1)$ if
   $\frac{1}{2}m_{2i}(\eta_i)<1<m_{1i}(\eta_i),$
which holds if $$0<\eta_i<\frac{e^{2\omega}-1}{e^{2\omega}+1},\q i=1,2.$$ From  Theorem \ref{thm4.2},  this leads to the existence of a positive $\omega$-periodic solution for the impulsive Nicholson system. This shows that  implementing a  small constant, $\omega$-periodic positive impulse to  system \eqref{Nplanar_Counter}, for any periodicity $\omega>0$, can create a positive periodic solution,  whereas  populations are otherwise driven to extinction. 
\end{exmp}

\section{Conclusions}
In the present paper, we  consider $\omega$-periodic delayed systems \eqref{sys}, with either discrete or distributed  delays and subject to $\omega$-periodic impulses. Under very general conditions on the nonli\-nea\-rities and impulses, we prove that \eqref{sys} possesses at least one positive $\omega$-periodic solution, by using Krasnoselskii's fixed point theorem. As far as the authors know, this is one of the first papers pro\-ving the existence of positive periodic solutions for systems of  differential equations with delays and impulses.
Moreover, the original method proposed here applies to very broad classes of impulsive systems of DDEs under very mild assumptions on the impulses, which
are in general nonlinear and whose signs may vary.  In fact, recently Zhang et al. \cite{ZHW} studied the  particular planar Nicholson system \eqref{Nplanar_Imp} with {\it linear} impulses $I_{ik}(u)=\eta_{ik}u$, where the constants $\eta_{ik}>-1$  are subject to the additional restriction that the functions $t\mapsto \prod_{k:t_k\in[0,t)}(1+\eta_{ik}) \, (i=1,2)$  are $\omega$-periodic. Contrary to the authors' claim  however,  condition \eqref{hypZHW} is not sufficient to guarantee the existence of a positive periodic solution, as shown in Example \ref{exmp4.4}.

The major novelty of our approach is based on the particular operator  whose fixed points are the periodic solutions we are looking for.
The construction of such an operator  follows along the main lines  in \cite{BF,FO19}, however most of the arguments have to be modified, due to the multidimensional character of \eqref{sys}.  
This operator   is far different from  other ones constructed in the literature, see e.g.~\cite{BCZ,Li_et.al,TangZou,Yan07,ZhangFeng,ZHW},
since it departs from inserting  the impulses  in a {\it multiplicative} way (rather than {\it additive}),  through the products of the auxiliary functions $J_{ik}(u)$ in \eqref{B&Js}. 

Our results are illustrated and analysed within the context of some related works, showing the advantage and novelty of our approach. We have restricted ourselves to the presentation of a few selected examples to reduce the size of this manuscript; many other examples could have been given,  e.g. multidimensional versions of the models treated in   \cite{FO19,Li_et.al,ZhangFeng}.   
Most of the applications refer to systems  with bounded delays, but our results apply with straightforward changes to impulsive systems \eqref{sys}  with infinite delay (see Remark \ref{rmk2.1}). Once the existence of a positive periodic solution is established, a future line of investigation is to study sufficient conditions for its global attractivity: of course, this depends strongly on the particular nonlinearities $g_i$ in \eqref{sys}, as shown in \cite{DingFu,Faria20,HWH,Wang+19} for nonimpulsive Nicholson systems.  


Although fixed point theorems in cones have been employed in some works, mostly dealing with periodic competitive  Lotka-Volterra systems of DDEs as in  \cite{BCZ,TangZou},  the literature for impulsive versions of   periodic multidimensional DDEs is almost nonexistent, so we believe that the  new results presented  here have signi\-fi\-cant outcomes, namely in addressing Nicholson-type systems.
The present technique has the potential to treat other families of impulsive systems with delay, such as Lotka-Volterra models, Nicholson systems with patch structure  and nonlinear mortality terms as mentioned in Remark \ref{rmk3.4}, or hematopoiesis systems with harvesting terms (see \cite{KongNietoFu} for a nonimpulsive very general model).


\section*{Acknowledgements} 
This work was  supported by  FCT-Funda\c c\~ao para a Ci\^encia e a Tecnologia (Portugal) under project UIDB/04561/2020 (T. Faria). The paper
 was partially written during the stay of R. Figueroa in Lisboa, thanks to a Jos\'e Castillejo 2016 grant of Ministerio de Educaci\'on, Cultura y Deporte, Government of Spain.

{\small

\end{document}